\documentclass[a4paper]{amsart}
\usepackage[margin=3.5cm]{geometry}
\usepackage{amsmath,amssymb,amsthm,amscd,amsfonts,comment,mathtools}
\usepackage{mathtools}

\usepackage[colorlinks=true, allcolors=blue]{hyperref}
\usepackage{tikz}
\usetikzlibrary{cd}

\newcommand\xc{1}
\newcommand\yc{.4}  
\newcommand\ycc{.3}   
\newcommand\magn{.7}
\newcommand{\xx}{0.4}   
\newcommand{\yy}{0.15}  

\DeclareMathOperator{\QIP}{QIP}
\DeclareMathOperator{\rk}{rk} 

\DeclareMathOperator{\Hom}{Hom}
\DeclareMathOperator{\Ext}{Ext}
\DeclareMathOperator{\Rep}{\mathrm{Rep}}
\DeclareMathOperator{\codim}{\mathrm{codim}}
\DeclareMathOperator{\GL}{\mathrm{GL}}
\DeclareMathOperator{\MM}{M}

\newcommand{\R}{\mathbb{R}}
\newcommand{\C}{\mathbb{C}}

\newcommand{\Z}{\mathbb{Z}}
\newcommand{\N}{\mathbb{N}}

\newcommand{\cP}{\mathcal{P}}
\newcommand{\OO}{\mathcal{O}}
\newcommand{\um}{\underline{m}}
\newcommand{\ud}{\underline{d}}
\newcommand{\ue}{\underline{e}}
\newcommand{\ur}{\underline{r}}
\newcommand{\uv}{\underline{v}}

\newcommand{\dv}{\underline{d}} 
\newcommand{\rv}{\underline{e}}
\newcommand{\rvector}{rising vector}
\newcommand{\Vrep}[1]{\operatorname{M}(#1)} 
\newcommand{\Orep}[1]{\OO(#1)} 
\newcommand{\Dset}[1]{D(#1)} 
\newcommand{\Irred}[1]{\MM_{#1}} 
\newcommand{\QIPF}{F} 
\newcommand{\One}{\mathbf{1}}

\newtheorem{theorem}{Theorem}[section]
\newtheorem{proposition}[theorem]{Proposition}
\newtheorem{lemma}[theorem]{Lemma}
\newtheorem{definition}[theorem]{Definition}
\newtheorem{corollary}[theorem]{Corollary}
 
\newtheorem{example}[theorem]{Example}
\newtheorem{remark}[theorem]{Remark}

\title{The main reasons for matrices multiplying to zero}

\author{Jakub Koncki}
\address{Institute of Mathematics, Polish Academy of Sciences, Poland}

\author{Rich\'ard Rim\'anyi}
\address{Department of Mathematics, UNC Chapel Hill, Chapel Hill, NC, USA}

\begin{document}

\begin{abstract}
   We provide an explicit description of the maximal-dimensional components of the variety parametrizing sequences of matrices of prescribed sizes whose product is zero.
\end{abstract}

\maketitle

\section{Introduction} 
For $2 \times 2$ matrices $A$ and $B$ there can be three reasons for $AB=0$. Either $A=0$ or $B=0$ or both have rank at most 1 and the kernel of $A$ contains the image of $B$. These are the three components of the subvariety $\{(A,B) \in \C^{2\times 2} \times \C^{2\times 2}:AB=0\}$. The codimensions of these components are 4, 4, and 3, respectively. Since 3 is the smallest codimension, we call the last component the `main reason' for $AB=0$. 

More generally, consider a sequence of matrices $A_1, A_2, \ldots, A_n$, with $A_k \in \C^{d_k \times d_{k-1}}$ for a sequence of dimensions $\ud = (d_0, d_1, \dots, d_n)$, and the variety $\Sigma_{\ud}=\{A_nA_{n-1}\ldots A_1=0\}$. Describing all components of $\Sigma_{\ud}$---even just their number and dimensions---is an ambitious problem. Instead, we focus on the maximal-dimensional components, whose codimension we denote by $C$ and whose number we denote by $\theta$. 

Recently, \cite{LR} provided explicit formulas for $C$ and $\theta$ for any dimension vector $\ud$, motivated by applications to the {\em learning coefficient of deep linear networks} in the context of Singular Learning Theory \cite{watanabe2009, watanabe2024recent}. These formulas will be recalled in Section~\ref{sec:LR results}. That work also described the maximal-dimensional components explicitly in a special case: when $\ud$ is weakly increasing or decreasing.  

In this paper, we provide an explicit description of the maximal-dimensional components of $\Sigma_{\ud}$ for arbitrary $\ud$, thereby identifying the {\em main reasons} for matrices multiplying to zero. Our approach is  as follows:  (1) We formulate a quadratic integer program. (2) We describe an algorithm that, for each of the $\theta$ optimal solutions, constructs a {\em lace diagram}, a picture of dots and lines. These $\theta$ diagrams encode the $\theta$ maximal-dimensional components of $\Sigma_{\ud}$ in two ways: the diagrams represents general elements of the components, and they also encode their defining equations.   

For example, for $\ud=(2,3,2,3)$ ($C=4$, $\theta=3$) 
the integer program is: 
\[
\min(e_1e_2+e_1e_4+e_2e_4+e_2+e_4), \text{  subject to } e_i\in \N, e_1+e_2+e_4=2.
\]
To the three optimal solutions $(e_1,e_2,e_4)=(1,0,1), (1,1,0), (2,0,0)$ we associate the lace diagrams  
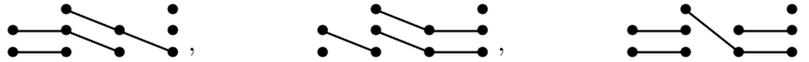
\[
\begin{tikzpicture}[baseline=.5,scale=\magn]
  \node (a1) at (0*\xc,0*\yc){$\bullet$}; 
  \node (a2) at (0*\xc,1*\yc){$\bullet$}; 

  \node (b1) at (1*\xc,0*\yc){$\bullet$};
  \node (b2) at (1*\xc,1*\yc){$\bullet$}; 
  \node (b3) at (1*\xc,2*\yc){$\bullet$}; 

  \node (c1) at (2*\xc,0*\yc){$\bullet$}; 
  \node (c2) at (2*\xc,1*\yc){$\bullet$}; 

  \node (d1) at (3*\xc,0*\yc){$\bullet$}; 
  \node (d2) at (3*\xc,1*\yc){$\bullet$}; 
  \node (d3) at (3*\xc,2*\yc){$\bullet$}; 

\draw[thick] (a1.center) -- (b1.center); 
\draw[thick] (a2.center) -- (b2.center);
\draw[thick] (b2.center) -- (c1.center); 
\draw[thick] (b3.center) -- (d1.center); 
\end{tikzpicture},
\qquad\qquad
\begin{tikzpicture}[baseline=.5,scale=\magn]
  \node (a1) at (0*\xc,0*\yc){$\bullet$}; 
  \node (a2) at (0*\xc,1*\yc){$\bullet$}; 

  \node (b1) at (1*\xc,0*\yc){$\bullet$};
  \node (b2) at (1*\xc,1*\yc){$\bullet$}; 
  \node (b3) at (1*\xc,2*\yc){$\bullet$}; 

  \node (c1) at (2*\xc,0*\yc){$\bullet$}; 
  \node (c2) at (2*\xc,1*\yc){$\bullet$}; 

  \node (d1) at (3*\xc,0*\yc){$\bullet$}; 
  \node (d2) at (3*\xc,1*\yc){$\bullet$}; 
  \node (d3) at (3*\xc,2*\yc){$\bullet$}; 

\draw[thick] (a2.center) -- (b1.center); 
\draw[thick] (b2.center) -- (c1.center);
\draw[thick] (b3.center) -- (c2.center); 
\draw[thick] (c1.center) -- (d1.center); 
\draw[thick] (c2.center) -- (d2.center);
\end{tikzpicture},
\qquad\qquad
\begin{tikzpicture}[baseline=.5,scale=\magn]
  \node (a1) at (0*\xc,0*\yc){$\bullet$}; 
  \node (a2) at (0*\xc,1*\yc){$\bullet$}; 

  \node (b1) at (1*\xc,0*\yc){$\bullet$};
  \node (b2) at (1*\xc,1*\yc){$\bullet$}; 
  \node (b3) at (1*\xc,2*\yc){$\bullet$}; 

  \node (c1) at (2*\xc,0*\yc){$\bullet$}; 
  \node (c2) at (2*\xc,1*\yc){$\bullet$}; 

  \node (d1) at (3*\xc,0*\yc){$\bullet$}; 
  \node (d2) at (3*\xc,1*\yc){$\bullet$}; 
  \node (d3) at (3*\xc,2*\yc){$\bullet$}; 

\draw[thick] (a1.center) -- (b1.center); 
\draw[thick] (a2.center) -- (b2.center); 
\draw[thick] (b3.center) -- (c1.center); 
\draw[thick] (c1.center) -- (d1.center);
\draw[thick] (c2.center) -- (d2.center); 
\end{tikzpicture}.
\]
The first one encodes the component 
\begin{multline*}
\{ (A_1,A_2,A_3)\in \C^{3\times 2} \times \C^{2\times 3} \times \C^{3\times 2}
\ :\  \\   
\rk(A_1)\leq 2, \rk(A_2)\leq 2, \rk(A_3)\leq 1,
\rk(A_2A_1)\leq 1, \rk(A_3A_2)\leq 1, \rk(A_3A_2A_1)=0\}
\\
=\{ (A_1,A_2,A_3)\in \C^{3\times 2} \times \C^{2\times 3} \times \C^{3\times 2}
\ :\ 
\rk(A_3)\leq 1, \rk(A_2A_1)\leq 1
\}
\end{multline*}
of $\Sigma_{(2,3,2,3)}$. A general element of this component is
\[
A_1=\begin{bmatrix}
  0 & 0 \\ 1 & 0 \\ 0 & 1    
\end{bmatrix},
\quad
A_2=\begin{bmatrix}
  1 & 0 & 0 \\ 0 & 1  & 0   
\end{bmatrix},
\quad
A_3=\begin{bmatrix}
  0 & 0 \\ 0 & 0 \\ 1 & 0   
\end{bmatrix},
\]
in the sense that the component is the closure of an orbit of $(A_1,A_2,A_3)$ for a certain group action. 

Let us remark that a naive approach can yield a similar description of the sought maximal-dimensional components---namely, by minimizing the orbit-codimension formula over a set known as Kostant partitions of~$\ud$. However, this optimization takes place within a complicated polytope in $\R^{\binom{n}{2}}$. Our refinement reduces the problem to an integer program over a simplex in dimension $n$. Furthermore, as noted in \cite{LR}, this smaller integer program reveals the a priori surprising fact that $C$ and $\theta$ remain invariant under permutations of the components of $\ud$.

\bigskip

\noindent{\bf Acknowledgments.} JK was supported by National Science Centre (Poland) grant Sonatina 2023/48/C/ST1/00002. RR was supported by the National Science Foundation under Grant No. 2200867. Special thanks to S.~P.~Lehalleur for helpful discussions on the topic.   

\section{Equioriented type A quiver representations and their orbits}

\subsection{The quiver representation}
Let $n$ be a non-negative integer, and $\ud=(d_0,d_1,\ldots,d_n)\in \N^{n+1}$, the dimension vector. We consider the quiver representation space
\begin{equation}\label{eq:quiver HOM}
\Rep_{\ud}=\bigoplus_{i=1}^n \Hom(\C^{d_{i-1}},\C^{d_{i}}),
\end{equation}
with the natural reparametrization action of $\GL_{\ud}=\bigtimes_{i=0}^n \GL(d_i,\C)$. 

\subsection{Orbits}
Now we recall the relevant special case of Gabriel's theorem \cite{gabriel_original}, \cite[Ch.3]{kirillov} that describes the finitely many orbits of this action, via representatives and combinatorial codes.

For $0\leq k \leq l\leq n$ let $\One_{kl}$ denote the integer vector $(x_0,\ldots,x_n)$ with $x_i=1$ if $k\leq i\leq l$, and $x_i=0$ otherwise.

\begin{definition}
\begin{enumerate}
    \item 
    We call $\um=(m_{kl})_{0\leq k \leq l\leq n}$ a {\em Kostant partition} for the dimension vector $\ud$ if $\sum_{kl} m_{kl}\One_{kl}=\ud$.
    \item We call $\ur=(r_{ij})_{0\leq i \leq j\leq n}$ the {\em rank pattern} of the Kostant partition $\um$ if \[r_{ij}=\sum_{k\leq i\leq j\leq l} m_{kl}.\]  
    \item \cite{abeasis-del_fra:equioriented, BFRpositivity}
    An arrangement of dots and lines in the plane is called a {\em lace diagram}, if the dots are arranged in columns at horizontal positions $0$ to $n$; in column $k$ there are $d_k$ dots in consecutive integer positions; and some segments are drawn between a dot in column $k$ and a dot in column $k+1$ in such a way that a dot can be adjacent to at most one segment to the left and at most one to the right. These lines glue together to maximal $[k,l]$ intervals. If the number of $[k,l]$ intervals is $m_{kl}$ we say that the lace diagram represents the Kostant partition $\um$. 
\end{enumerate}
\end{definition}

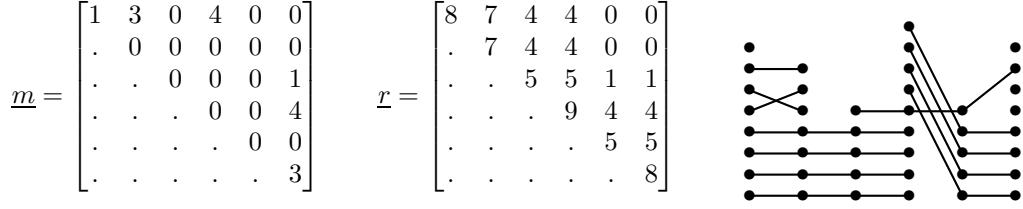
\begin{figure}
\[
\um=
\begin{bmatrix}
1 & 3 & 0 & 4 & 0 & 0 \\
. & 0 & 0 & 0 & 0 & 0 \\
. & . & 0 & 0 & 0 & 1 \\
. & . & . & 0 & 0 & 4 \\
. & . & . & . & 0 & 0 \\
. & . & . & . & . & 3 
\end{bmatrix}
\qquad
\ur=
\begin{bmatrix}
8 & 7 & 4 & 4 & 0 & 0 \\
. & 7 & 4 & 4 & 0 & 0 \\
. & . & 5 & 5 & 1 & 1 \\
. & . & . & 9 & 4 & 4 \\
. & . & . & . & 5 & 5 \\
. & . & . & . & . & 8 
\end{bmatrix}
\qquad
\begin{tikzpicture}[baseline=35,scale=\magn]
  \node (a1) at (0*\xc,0*\yc){$\bullet$}; 
  \node (a2) at (0*\xc,1*\yc){$\bullet$}; 
  \node (a3) at (0*\xc,2*\yc){$\bullet$}; 
  \node (a4) at (0*\xc,3*\yc){$\bullet$}; 
  \node (a5) at (0*\xc,4*\yc){$\bullet$}; 
  \node (a6) at (0*\xc,5*\yc){$\bullet$}; 
  \node (a7) at (0*\xc,6*\yc){$\bullet$}; 
  \node (a8) at (0*\xc,7*\yc){$\bullet$}; 

  \node (b1) at (1*\xc,0*\yc){$\bullet$};
  \node (b2) at (1*\xc,1*\yc){$\bullet$}; 
  \node (b3) at (1*\xc,2*\yc){$\bullet$}; 
  \node (b4) at (1*\xc,3*\yc){$\bullet$}; 
  \node (b5) at (1*\xc,4*\yc){$\bullet$}; 
  \node (b6) at (1*\xc,5*\yc){$\bullet$}; 
  \node (b7) at (1*\xc,6*\yc){$\bullet$};

  \node (c1) at (2*\xc,0*\yc){$\bullet$}; 
  \node (c2) at (2*\xc,1*\yc){$\bullet$}; 
  \node (c3) at (2*\xc,2*\yc){$\bullet$}; 
  \node (c4) at (2*\xc,3*\yc){$\bullet$}; 
  \node (c5) at (2*\xc,4*\yc){$\bullet$}; 

  \node (d1) at (3*\xc,0*\yc){$\bullet$}; 
  \node (d2) at (3*\xc,1*\yc){$\bullet$}; 
  \node (d3) at (3*\xc,2*\yc){$\bullet$}; 
  \node (d4) at (3*\xc,3*\yc){$\bullet$}; 
  \node (d5) at (3*\xc,4*\yc){$\bullet$}; 
  \node (d6) at (3*\xc,5*\yc){$\bullet$}; 
  \node (d7) at (3*\xc,6*\yc){$\bullet$}; 
  \node (d8) at (3*\xc,7*\yc){$\bullet$}; 
  \node (d9) at (3*\xc,8*\yc){$\bullet$}; 

  \node (e1) at (4*\xc,0*\yc){$\bullet$}; 
  \node (e2) at (4*\xc,1*\yc){$\bullet$}; 
  \node (e3) at (4*\xc,2*\yc){$\bullet$}; 
  \node (e4) at (4*\xc,3*\yc){$\bullet$}; 
  \node (e5) at (4*\xc,4*\yc){$\bullet$}; 

  \node (f1) at (5*\xc,0*\yc){$\bullet$}; 
  \node (f2) at (5*\xc,1*\yc){$\bullet$}; 
  \node (f3) at (5*\xc,2*\yc){$\bullet$}; 
  \node (f4) at (5*\xc,3*\yc){$\bullet$}; 
  \node (f5) at (5*\xc,4*\yc){$\bullet$}; 
  \node (f6) at (5*\xc,5*\yc){$\bullet$}; 
  \node (f7) at (5*\xc,6*\yc){$\bullet$}; 
  \node (f8) at (5*\xc,7*\yc){$\bullet$};

\draw[thick] (a1.center) -- (d1.center); 
\draw[thick] (a2.center) -- (d2.center); 
\draw[thick] (a3.center) -- (d3.center); 
\draw[thick] (a4.center) -- (d4.center); 
\draw[thick] (a5.center) -- (b6.center); 
\draw[thick] (a6.center) -- (b5.center); 
\draw[thick] (a7.center) -- (b7.center); 
\draw[thick] (c5.center) -- (e5.center); \draw[thick] (e5.center) -- (f7.center); 
\draw[thick] (d6.center) -- (e1.center); 
\draw[thick] (d7.center) -- (e2.center); 
\draw[thick] (d8.center) -- (e3.center); 
\draw[thick] (d9.center) -- (e4.center); 
\draw[thick] (e1.center) -- (f1.center); 
\draw[thick] (e2.center) -- (f2.center); 
\draw[thick] (e3.center) -- (f3.center); 
\draw[thick] (e4.center) -- (f4.center); 
\end{tikzpicture}
\]
\caption{A Kostant partition $\um$ for the dimension vector $\ud=(8,7,5,9,5,8)$, the induced rank pattern $\ur$ (row and column indices run from 0 to $n$), and a lace diagram representing $\um$. Another lace diagram representing the same $\um$ is the top-left picture in Figure~\ref{fig:875958 raising vectors}.}
\label{fig:m-r-lace}
\end{figure}

Gabriel's theorem, applied to the equioriented type A quiver, implies that the orbits of the representation~\eqref{eq:quiver HOM} are in bijection with the Kostant partitions for $\ud$. In particular, the representation has finitely many orbits for any $\ud$. The one associated with the Kostant partition $\um$ is called $\OO_{\um}$. Defining equations for the closure of $\OO_{\um}$, as well as representatives of $\OO_{\um}$ can be described through the associated $\ur$ and lace diagrams:
\begin{itemize}
\item 
Let $\ur$ be the rank pattern of $\um$. We have
\[
\overline{\OO}_{\um}
=
\{ 
(A_i)\in\Rep_{\ud}
\ :\ 
\rk( A_l  A_{l-1}  \ldots  A_{k+1})\leq r_{kl}
\ \forall k \leq l
\}.
\]
\item Consider a lace diagram, and choose an identification of the dots in the $k$'th column with the standard basis vectors $v_i$ of $\C^{d_k}$. Let $A_k:\C^{d_{k-1}}\to \C^{d_k}$ be the linear map that maps $v_i \in \C^{d_{k-1}}$ to $w\in \C^{d_k}$ if there is a segment connecting $v_i$ with $w$. If there is no segment coming out of $v_i$ to the right, then that basis vector is mapped to 0. The obtained element $(A_k) \in \Rep_{\ud}$, a sequence of 01-matrices, is in $\OO_{\um}$.
\end{itemize}

\subsection{Quiver modules}

In the proof of Gabriel's theorem, as well as in our arguments below, the language of quiver modules is important. A quiver module (for our quiver) is a collection of finite dimensional vector spaces $V_i$ for $i=0,\ldots,n$ and linear maps $\phi_i:V_{i-1}\to V_i$. The sequence of dimensions is called the dimension vector of the module.

The concepts of homomorphism, isomorphism, direct sum $\oplus$, $\Hom$, $\Ext$,  direct sum decomposition, indecomposable modules are defined the usual/obvious way for quiver representations, see \cite{kirillov}. The Krull-Schmidt theorem about the existence and uniqueness of decomposition to indecomposable summands holds.

An algebraic statement essentially equivalent to the geometric statement on the orbit description above is that the indecomposable quiver modules are 
\[
\MM_{kl}= ( 0 \to 0 \to \ldots \to 0 \to 
\C \xrightarrow{=} \C \xrightarrow{=}\ldots \xrightarrow{=} \C 
\to 0 \to 0 \to \ldots \to 0)
\]
where the first and last $\C$ are in positions $k \leq l$. An obvious consequence is that isomorphism classes of quiver modules with dimension vector $\ud$  are in bijection with Kostant partitions $\um$ of $\ud$. A module correcponding to $\um$ is $\MM_{\um}=\oplus_{kl} m_{kl} \MM_{kl}$.

\subsection{Voight lemma}
The bijections $\um \leftrightarrow \OO_{\um} \leftrightarrow \MM_{\um}$ are at the heart of the interplay between quiver combinatorics, geometry, and algebra. A fundamental example for this interplay is 

\begin{lemma} \label{lem:voight} \cite{voight}
A normal slice to $\OO_{\um}$ at a point is isomorphic to $\Ext(\MM_{\um},\MM_{\um})$.
\end{lemma}

As a consequence we have a formula for the codimension of the orbits: 
\begin{multline*}
\codim(\OO_{\um} \subset \Rep_{\ud})
=
\dim \Ext(\MM_{\um},\MM_{\um})
=
\dim \Ext( \oplus_{ij} m_{ij}\MM_{ij}, \oplus_{uv} m_{uv}\MM_{uv})
\\
=
\sum_{ij}\sum_{uv} m_{ij}m_{uv}\dim \Ext(\MM_{ij},\MM_{uv}),
\end{multline*}
where an easy calculation (cf. \cite[Lemma~4.4]{FRduke}) gives 
\begin{equation}\label{eqn:ExtIndecomposables}
\dim \Ext(\MM_{ij},\MM_{uv})=
\begin{cases}
    1 & \text{if } i+1\leq u \leq j+1 \leq v, \\
    0 & \text{otherwise.}
\end{cases}
\end{equation}

\begin{corollary} \label{cor:Ext}
	For an arbitrary quiver module $\MM$ and integer $k\in[0,n]$ we have
	\[
    \Ext(\MM,\Irred{0k})=0=\Ext(\Irred{kn},\MM).
    \]
\end{corollary}

The following less known proposition is also a corollary of the codimension formula, although it requires more work. 

\begin{proposition} \label{tw:OpenOrbit} \cite[Cor.~2, Rem.~1.]{BR}
    Consider the lace diagram for $\ud$ obtained by aligning the columns of dots at the bottom, and drawing all possible horizontal segments (see the picture below). The associated orbit is the open orbit of $\Rep_{\ud}$. \qed
\end{proposition}

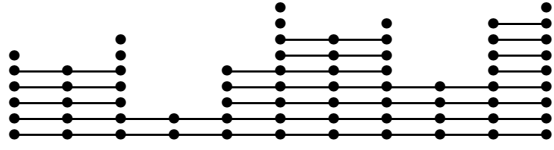
\[
\begin{tikzpicture}[scale=\magn]
  \node (a1) at (0*\xc,0*\ycc){$\bullet$}; 
  \node (a2) at (0*\xc,1*\ycc){$\bullet$}; 
  \node (a3) at (0*\xc,2*\ycc){$\bullet$}; 
  \node (a4) at (0*\xc,3*\ycc){$\bullet$}; 
  \node (a5) at (0*\xc,4*\ycc){$\bullet$}; 
  \node (a6) at (0*\xc,5*\ycc){$\bullet$};

  \node (b1) at (1*\xc,0*\ycc){$\bullet$}; 
  \node (b2) at (1*\xc,1*\ycc){$\bullet$}; 
  \node (b3) at (1*\xc,2*\ycc){$\bullet$}; 
  \node (b4) at (1*\xc,3*\ycc){$\bullet$}; 
  \node (b5) at (1*\xc,4*\ycc){$\bullet$}; 

  \node (c1) at (2*\xc,0*\ycc){$\bullet$}; 
  \node (c2) at (2*\xc,1*\ycc){$\bullet$}; 
  \node (c3) at (2*\xc,2*\ycc){$\bullet$}; 
  \node (c4) at (2*\xc,3*\ycc){$\bullet$}; 
  \node (c5) at (2*\xc,4*\ycc){$\bullet$}; 
  \node (c6) at (2*\xc,5*\ycc){$\bullet$}; 
  \node (c7) at (2*\xc,6*\ycc){$\bullet$}; 

  \node (d1) at (3*\xc,0*\ycc){$\bullet$}; 
  \node (d2) at (3*\xc,1*\ycc){$\bullet$}; 

  \node (e1) at (4*\xc,0*\ycc){$\bullet$}; 
  \node (e2) at (4*\xc,1*\ycc){$\bullet$}; 
  \node (e3) at (4*\xc,2*\ycc){$\bullet$}; 
  \node (e4) at (4*\xc,3*\ycc){$\bullet$}; 
  \node (e5) at (4*\xc,4*\ycc){$\bullet$}; 

  \node (f1) at (5*\xc,0*\ycc){$\bullet$}; 
  \node (f2) at (5*\xc,1*\ycc){$\bullet$}; 
  \node (f3) at (5*\xc,2*\ycc){$\bullet$}; 
  \node (f4) at (5*\xc,3*\ycc){$\bullet$}; 
  \node (f5) at (5*\xc,4*\ycc){$\bullet$}; 
  \node (f6) at (5*\xc,5*\ycc){$\bullet$}; 
  \node (f7) at (5*\xc,6*\ycc){$\bullet$}; 
  \node (f8) at (5*\xc,7*\ycc){$\bullet$}; 
  \node (f9) at (5*\xc,8*\ycc){$\bullet$}; 

  \node (g1) at (6*\xc,0*\ycc){$\bullet$}; 
  \node (g2) at (6*\xc,1*\ycc){$\bullet$}; 
  \node (g3) at (6*\xc,2*\ycc){$\bullet$}; 
  \node (g4) at (6*\xc,3*\ycc){$\bullet$}; 
  \node (g5) at (6*\xc,4*\ycc){$\bullet$}; 
  \node (g6) at (6*\xc,5*\ycc){$\bullet$}; 
  \node (g7) at (6*\xc,6*\ycc){$\bullet$}; 

  \node (h1) at (7*\xc,0*\ycc){$\bullet$}; 
  \node (h2) at (7*\xc,1*\ycc){$\bullet$}; 
  \node (h3) at (7*\xc,2*\ycc){$\bullet$}; 
  \node (h4) at (7*\xc,3*\ycc){$\bullet$}; 
  \node (h5) at (7*\xc,4*\ycc){$\bullet$}; 
  \node (h6) at (7*\xc,5*\ycc){$\bullet$}; 
  \node (h7) at (7*\xc,6*\ycc){$\bullet$}; 
  \node (h8) at (7*\xc,7*\ycc){$\bullet$}; 

  \node (i1) at (8*\xc,0*\ycc){$\bullet$}; 
  \node (i2) at (8*\xc,1*\ycc){$\bullet$}; 
  \node (i3) at (8*\xc,2*\ycc){$\bullet$}; 
  \node (i4) at (8*\xc,3*\ycc){$\bullet$}; 
 
  \node (j1) at (9*\xc,0*\ycc){$\bullet$}; 
  \node (j2) at (9*\xc,1*\ycc){$\bullet$}; 
  \node (j3) at (9*\xc,2*\ycc){$\bullet$}; 
  \node (j4) at (9*\xc,3*\ycc){$\bullet$}; 
  \node (j5) at (9*\xc,4*\ycc){$\bullet$}; 
  \node (j6) at (9*\xc,5*\ycc){$\bullet$}; 
  \node (j7) at (9*\xc,6*\ycc){$\bullet$}; 
  \node (j8) at (9*\xc,7*\ycc){$\bullet$}; 

  \node (k1) at (10*\xc,0*\ycc){$\bullet$}; 
  \node (k2) at (10*\xc,1*\ycc){$\bullet$}; 
  \node (k3) at (10*\xc,2*\ycc){$\bullet$}; 
  \node (k4) at (10*\xc,3*\ycc){$\bullet$}; 
  \node (k5) at (10*\xc,4*\ycc){$\bullet$}; 
  \node (k6) at (10*\xc,5*\ycc){$\bullet$}; 
  \node (k7) at (10*\xc,6*\ycc){$\bullet$}; 
  \node (k8) at (10*\xc,7*\ycc){$\bullet$}; 
  \node (k9) at (10*\xc,8*\ycc){$\bullet$}; 

\draw[thick] (a1.center) -- (k1.center);
\draw[thick] (a2.center) -- (k2.center);
\draw[thick] (a3.center) -- (c3.center);
\draw[thick] (e3.center) -- (k3.center);
\draw[thick] (a4.center) -- (c4.center);
\draw[thick] (e4.center) -- (k4.center);
\draw[thick] (a5.center) -- (c5.center);
\draw[thick] (e5.center) -- (h5.center);
\draw[thick] (j5.center) -- (k5.center);
\draw[thick] (f6.center) -- (h6.center);
\draw[thick] (j6.center) -- (k6.center);
\draw[thick] (f7.center) -- (h7.center);
\draw[thick] (j7.center) -- (k7.center);
\draw[thick] (j8.center) -- (k8.center);

\end{tikzpicture}
\]
\begin{corollary} \label{cor:OpenOrbit}
    Let $\MM$ be a quiver module associated to a lace diagram such as in the proposition above. Then
    $\Ext(\MM,\MM)=0$.
\end{corollary}

\section{$\Sigma_{\ud}$ and its components}

\subsection{The protagonist}

\begin{definition}
    For a dimension vector $\ud$ define the subvariety consisting of tuples of matrices multiplying to 0:
    \[
    \Sigma_{\ud}=\{ (A_i)\in \Rep_{\ud}\ :\ A_n A_{n-1} \ldots A_1=0\}
    \subset 
    \Rep_{\ud}.
    \]
\end{definition}
Clearly $\Sigma_{\ud}$ is an invariant subset, hence it is a union of orbits. In general, as a variety, it has many components. Here are some examples
\begin{itemize}
    \item $\Sigma_{(2,2,2)}$ is the collection of pairs of $2\times 2$ matrices $(A,B)$ for which $BA=0$. This variety has three components: (i) $\{A=0\}$ (of codimension 4), (ii) $\{B=0\}$ (of codimension 4), (iii) $\{\det(A)=\det(B)=0,BA=0\}$ (of codimension~3).
    \item Let $\ud=(5,5,6,6,6,6)$. The variety $\Sigma_{\ud}$ has five codimension 19 components and many smaller dimensional components. One of the top dimensional components is the closure of the orbit $\OO_{\um}$ with 
    \[
    \um=
\begin{bmatrix}
1 & 1 & 1 & 1 & 1 & 0 \\
. & 0 & 0 & 0 & 0 & 1 \\
. & . & 0 & 0 & 0 & 2 \\
. & . & . & 0 & 0 & 1 \\
. & . & . & . & 0 & 1 \\
. & . & . & . & . & 1 
\end{bmatrix}.
    \]
    \item For $\ud=(8,7,5,9,5,8)$ the variety $\Sigma_{\ud}$ has four maximal-dimensional components, one of them is the closure of the orbit described in Figure~\ref{fig:m-r-lace}. All four are also described in Figure~\ref{fig:875958 raising vectors}.
\end{itemize}

\begin{definition}
    Let $C=C_{\ud}$ denote the codimension of the maximal-dimensional components of $\Sigma_{\ud}$. Let $\theta=\theta_{\ud}$ denote the number of $C$ codimensional components.
\end{definition}

The maximal-dimensional components of $\Sigma_{\ud}$ are what we alluded to in the Introduction as the {\em main reasons} for a $\ud$-sequence of matrices multiplying to 0. Hence, $\theta$ answers the question: how many main reasons are there? We will recall the known answer in the next section. Then, in Section~\ref{sec:results} we will solve the problem of describing the $\theta$ maximal-dimensional components of $\Sigma_{\ud}$.

\subsection{Known results on $C$, $\theta$} \label{sec:LR results}
Theorems \ref{thm:LR1}, \ref{thm:LR2}, and \ref{thm:LR3} summarize the results about $C$ and $\theta$ from \cite{LR}.

\begin{theorem} \label{thm:LR1}
We have
\[
    \sum_{s=0}^{\min \ud} (-1)^s q^{\binom{s}{2}} \cP_s\cP_{\ud-s}= \theta q^ C + \text{higher order terms},
\]
where $\cP_s=1/((1-q)(1-q^2)\ldots (1-q^s))$ is the inverse Pochhammer symbol, $\cP_{\uv}=\prod \cP_{v_i}$, and $\ud-s$ means the vector obtained by subtracting $s$ from all components of $\ud$. \qed
\end{theorem}

For the dimension vector $\ud=(d_0,d_1,\ldots,d_n)$ let $\ud'=(d'_0\leq d'_1\leq \ldots \leq d'_n)$ be the same multiset, but arranged (weakly) increasingly. Consider $n$-tuples of non-negative integers $\ue=(e_1,e_2,\ldots,e_n)$ and the quadratic integer program $\QIP=\QIP(\ud)$ for them:
\begin{align}
\tag{QIP} \label{eq:QIP}
& \begin{array}{rl} \min & \displaystyle G_{\ud}(\ue)=\sum_{1\leq j\leq i\leq n} e_i(e_j+d'_j-d'_{j-1}) \\[.1in] \text{s.t.} & \displaystyle \ e_i\in \N, \quad \sum_{i=1}^n e_i=d'_0.
\end{array}
\end{align}

\begin{theorem} \label{thm:LR2}
    The minimal value of~\eqref{eq:QIP} is $C$, and the number of times~\eqref{eq:QIP} takes its minimal value is $\theta$. \qed
\end{theorem}    

\def\oldm{\tilde{n}}

\begin{theorem} \label{thm:LR3}
Define $\oldm=\max \left\{l\ |\ \sum_{i=0}^l d'_i \geq l d'_l\right\} \in \{1,2,\ldots,n\}$ and 
$S=\sum_{i=0}^{\oldm}d'_i$.
    We have
    \begin{align*} 
    C= & \frac{\oldm}{2}\left\{\frac{{S}}{\oldm} \right\}\left(1-\left\{\frac{{S}}{\oldm} \right\}\right) - \frac{\oldm(\oldm-1)}{2}\left(\frac{{S}}{\oldm}\right)^2 + \sum_{0\leq i<j\leq \oldm}d'_id'_j, \\
    \theta=&\binom{\oldm}{{S}-\oldm\left\lfloor \frac{{S}}{\oldm}+\frac{1}{2}\right\rfloor},
    \end{align*} 
where $\lfloor x \rfloor\in \Z$ is the integer part (`floor') of the real number $x$, and $\{x\}=x-\lfloor x \rfloor \in [0,1)$ is its fractional part. \qed
\end{theorem}

All three expressions for $C$ and $\theta$ are explicit and well programmable. Their proofs in \cite{LR} build on each other, in this order. The first one already displays the a priori surprising fact that both $C$ and $\theta$ are independent of the permutation of the components of~$\ud$.

\subsection{Description of the maximal-dimensional components, in a special case}
Let us assume that the components of $\ud$ are non-decreasing: $\ud=(d_0\leq d_1\leq \ldots \leq d_n)$. Let $\ue=(e_1,e_2,\ldots,e_n)$ be an optimal solution of the integer program~\eqref{eq:QIP} (recall that we have $\theta$ such choices for $\ue$). We are going to associate a lace diagram to $\ud$ and $\ue$ as follows:
\begin{itemize}
    \item Arrange $d_k$ dots in the $k$'th column so that they are aligned at their bottom row.
    \item Draw all possible horizontal segments between these dots.
    \item Delete the segment between the 0th and 1st columns in the bottom $e_1$ rows. Delete the segment between the 1st and 2nd columns in the next $e_2$ rows, etc. (Recall that $\sum e_i=d_0=\min(\ud)$, hence this way we deleted a segment in {\em all} of the rows that spanned from column $0$ to column $n$ after the second step of the algorithm.) 
\end{itemize}

\begin{figure}
\[
\begin{tikzpicture}[scale=\magn]
  \node (a1) at (0*\xc,0*\yc){$\bullet$}; 
  \node (a2) at (0*\xc,1*\yc){$\bullet$}; 
  \node (a3) at (0*\xc,2*\yc){$\bullet$}; 
  \node (a4) at (0*\xc,3*\yc){$\bullet$}; 
  \node (a5) at (0*\xc,4*\yc){$\bullet$};

  \node (b1) at (1*\xc,0*\yc){$\bullet$}; 
  \node (b2) at (1*\xc,1*\yc){$\bullet$}; 
  \node (b3) at (1*\xc,2*\yc){$\bullet$}; 
  \node (b4) at (1*\xc,3*\yc){$\bullet$}; 
  \node (b5) at (1*\xc,4*\yc){$\bullet$};

  \node (c1) at (2*\xc,0*\yc){$\bullet$}; 
  \node (c2) at (2*\xc,1*\yc){$\bullet$}; 
  \node (c3) at (2*\xc,2*\yc){$\bullet$}; 
  \node (c4) at (2*\xc,3*\yc){$\bullet$}; 
  \node (c5) at (2*\xc,4*\yc){$\bullet$}; 
  \node (c6) at (2*\xc,5*\yc){$\bullet$}; 
  \node (c7) at (2*\xc,6*\yc){$\bullet$};

  \node (d1) at (3*\xc,0*\yc){$\bullet$}; 
  \node (d2) at (3*\xc,1*\yc){$\bullet$}; 
  \node (d3) at (3*\xc,2*\yc){$\bullet$}; 
  \node (d4) at (3*\xc,3*\yc){$\bullet$}; 
  \node (d5) at (3*\xc,4*\yc){$\bullet$}; 
  \node (d6) at (3*\xc,5*\yc){$\bullet$}; 
  \node (d7) at (3*\xc,6*\yc){$\bullet$}; 
  \node (d8) at (3*\xc,7*\yc){$\bullet$};

  \node (e1) at (4*\xc,0*\yc){$\bullet$}; 
  \node (e2) at (4*\xc,1*\yc){$\bullet$}; 
  \node (e3) at (4*\xc,2*\yc){$\bullet$}; 
  \node (e4) at (4*\xc,3*\yc){$\bullet$}; 
  \node (e5) at (4*\xc,4*\yc){$\bullet$}; 
  \node (e6) at (4*\xc,5*\yc){$\bullet$}; 
  \node (e7) at (4*\xc,6*\yc){$\bullet$}; 
  \node (e8) at (4*\xc,7*\yc){$\bullet$};

  \node (f1) at (5*\xc,0*\yc){$\bullet$}; 
  \node (f2) at (5*\xc,1*\yc){$\bullet$}; 
  \node (f3) at (5*\xc,2*\yc){$\bullet$}; 
  \node (f4) at (5*\xc,3*\yc){$\bullet$}; 
  \node (f5) at (5*\xc,4*\yc){$\bullet$}; 
  \node (f6) at (5*\xc,5*\yc){$\bullet$}; 
  \node (f7) at (5*\xc,6*\yc){$\bullet$}; 
  \node (f8) at (5*\xc,7*\yc){$\bullet$}; 
  \node (f9) at (5*\xc,8*\yc){$\bullet$}; 

\draw[thick] (b1.center) -- (f1.center);
\draw[thick] (b2.center) -- (f2.center);
\draw[thick] (b3.center) -- (f3.center);
\draw[thick] (b4.center) -- (f4.center);

\draw[thick] (a5.center) -- (b5.center); \draw[thick] (c5.center) -- (f5.center);

\draw[thick] (c6.center) -- (f6.center);
\draw[thick] (c7.center) -- (f7.center);
\draw[thick] (d8.center) -- (f8.center);
\end{tikzpicture}
\qquad\qquad
\begin{tikzpicture}[scale=\magn]
  \node (a1) at (0*\xc,0*\yc){$\bullet$}; 
  \node (a2) at (0*\xc,1*\yc){$\bullet$}; 
  \node (a3) at (0*\xc,2*\yc){$\bullet$}; 
  \node (a4) at (0*\xc,3*\yc){$\bullet$}; 
  \node (a5) at (0*\xc,4*\yc){$\bullet$};

  \node (b1) at (1*\xc,0*\yc){$\bullet$}; 
  \node (b2) at (1*\xc,1*\yc){$\bullet$}; 
  \node (b3) at (1*\xc,2*\yc){$\bullet$}; 
  \node (b4) at (1*\xc,3*\yc){$\bullet$}; 
  \node (b5) at (1*\xc,4*\yc){$\bullet$};

  \node (c1) at (2*\xc,0*\yc){$\bullet$}; 
  \node (c2) at (2*\xc,1*\yc){$\bullet$}; 
  \node (c3) at (2*\xc,2*\yc){$\bullet$}; 
  \node (c4) at (2*\xc,3*\yc){$\bullet$}; 
  \node (c5) at (2*\xc,4*\yc){$\bullet$}; 
  \node (c6) at (2*\xc,5*\yc){$\bullet$}; 
  \node (c7) at (2*\xc,6*\yc){$\bullet$};

  \node (d1) at (3*\xc,0*\yc){$\bullet$}; 
  \node (d2) at (3*\xc,1*\yc){$\bullet$}; 
  \node (d3) at (3*\xc,2*\yc){$\bullet$}; 
  \node (d4) at (3*\xc,3*\yc){$\bullet$}; 
  \node (d5) at (3*\xc,4*\yc){$\bullet$}; 
  \node (d6) at (3*\xc,5*\yc){$\bullet$}; 
  \node (d7) at (3*\xc,6*\yc){$\bullet$}; 
  \node (d8) at (3*\xc,7*\yc){$\bullet$};

  \node (e1) at (4*\xc,0*\yc){$\bullet$}; 
  \node (e2) at (4*\xc,1*\yc){$\bullet$}; 
  \node (e3) at (4*\xc,2*\yc){$\bullet$}; 
  \node (e4) at (4*\xc,3*\yc){$\bullet$}; 
  \node (e5) at (4*\xc,4*\yc){$\bullet$}; 
  \node (e6) at (4*\xc,5*\yc){$\bullet$}; 
  \node (e7) at (4*\xc,6*\yc){$\bullet$}; 
  \node (e8) at (4*\xc,7*\yc){$\bullet$};

  \node (f1) at (5*\xc,0*\yc){$\bullet$}; 
  \node (f2) at (5*\xc,1*\yc){$\bullet$}; 
  \node (f3) at (5*\xc,2*\yc){$\bullet$}; 
  \node (f4) at (5*\xc,3*\yc){$\bullet$}; 
  \node (f5) at (5*\xc,4*\yc){$\bullet$}; 
  \node (f6) at (5*\xc,5*\yc){$\bullet$}; 
  \node (f7) at (5*\xc,6*\yc){$\bullet$}; 
  \node (f8) at (5*\xc,7*\yc){$\bullet$}; 
  \node (f9) at (5*\xc,8*\yc){$\bullet$}; 

\draw[thick] (b1.center) -- (f1.center);
\draw[thick] (b2.center) -- (f2.center);
\draw[thick] (b3.center) -- (f3.center);

\draw[thick] (a4.center) -- (b4.center); \draw[thick] (c4.center) -- (f4.center);
\draw[thick] (a5.center) -- (b5.center); \draw[thick] (c5.center) -- (f5.center);

\draw[thick] (c6.center) -- (f6.center);
\draw[thick] (c7.center) -- (f7.center);
\draw[thick] (d8.center) -- (f8.center);
\end{tikzpicture}
\]
\[
\begin{tikzpicture}[scale=\magn]
  \node (a1) at (0*\xc,0*\yc){$\bullet$}; 
  \node (a2) at (0*\xc,1*\yc){$\bullet$}; 
  \node (a3) at (0*\xc,2*\yc){$\bullet$}; 
  \node (a4) at (0*\xc,3*\yc){$\bullet$}; 
  \node (a5) at (0*\xc,4*\yc){$\bullet$};

  \node (b1) at (1*\xc,0*\yc){$\bullet$}; 
  \node (b2) at (1*\xc,1*\yc){$\bullet$}; 
  \node (b3) at (1*\xc,2*\yc){$\bullet$}; 
  \node (b4) at (1*\xc,3*\yc){$\bullet$}; 
  \node (b5) at (1*\xc,4*\yc){$\bullet$};

  \node (c1) at (2*\xc,0*\yc){$\bullet$}; 
  \node (c2) at (2*\xc,1*\yc){$\bullet$}; 
  \node (c3) at (2*\xc,2*\yc){$\bullet$}; 
  \node (c4) at (2*\xc,3*\yc){$\bullet$}; 
  \node (c5) at (2*\xc,4*\yc){$\bullet$}; 
  \node (c6) at (2*\xc,5*\yc){$\bullet$}; 
  \node (c7) at (2*\xc,6*\yc){$\bullet$};

  \node (d1) at (3*\xc,0*\yc){$\bullet$}; 
  \node (d2) at (3*\xc,1*\yc){$\bullet$}; 
  \node (d3) at (3*\xc,2*\yc){$\bullet$}; 
  \node (d4) at (3*\xc,3*\yc){$\bullet$}; 
  \node (d5) at (3*\xc,4*\yc){$\bullet$}; 
  \node (d6) at (3*\xc,5*\yc){$\bullet$}; 
  \node (d7) at (3*\xc,6*\yc){$\bullet$}; 
  \node (d8) at (3*\xc,7*\yc){$\bullet$};

  \node (e1) at (4*\xc,0*\yc){$\bullet$}; 
  \node (e2) at (4*\xc,1*\yc){$\bullet$}; 
  \node (e3) at (4*\xc,2*\yc){$\bullet$}; 
  \node (e4) at (4*\xc,3*\yc){$\bullet$}; 
  \node (e5) at (4*\xc,4*\yc){$\bullet$}; 
  \node (e6) at (4*\xc,5*\yc){$\bullet$}; 
  \node (e7) at (4*\xc,6*\yc){$\bullet$}; 
  \node (e8) at (4*\xc,7*\yc){$\bullet$};

  \node (f1) at (5*\xc,0*\yc){$\bullet$}; 
  \node (f2) at (5*\xc,1*\yc){$\bullet$}; 
  \node (f3) at (5*\xc,2*\yc){$\bullet$}; 
  \node (f4) at (5*\xc,3*\yc){$\bullet$}; 
  \node (f5) at (5*\xc,4*\yc){$\bullet$}; 
  \node (f6) at (5*\xc,5*\yc){$\bullet$}; 
  \node (f7) at (5*\xc,6*\yc){$\bullet$}; 
  \node (f8) at (5*\xc,7*\yc){$\bullet$}; 
  \node (f9) at (5*\xc,8*\yc){$\bullet$}; 

\draw[thick] (b1.center) -- (f1.center);
\draw[thick] (b2.center) -- (f2.center);
\draw[thick] (b3.center) -- (f3.center);

\draw[thick] (a4.center) -- (b4.center); \draw[thick] (c4.center) -- (f4.center);
\draw[thick] (a5.center) -- (c5.center); \draw[thick] (d5.center) -- (f5.center);

\draw[thick] (c6.center) -- (f6.center);
\draw[thick] (c7.center) -- (f7.center);
\draw[thick] (d8.center) -- (f8.center);
\end{tikzpicture}
\qquad\qquad
\begin{tikzpicture}[scale=\magn]
  \node (a1) at (0*\xc,0*\yc){$\bullet$}; 
  \node (a2) at (0*\xc,1*\yc){$\bullet$}; 
  \node (a3) at (0*\xc,2*\yc){$\bullet$}; 
  \node (a4) at (0*\xc,3*\yc){$\bullet$}; 
  \node (a5) at (0*\xc,4*\yc){$\bullet$};

  \node (b1) at (1*\xc,0*\yc){$\bullet$}; 
  \node (b2) at (1*\xc,1*\yc){$\bullet$}; 
  \node (b3) at (1*\xc,2*\yc){$\bullet$}; 
  \node (b4) at (1*\xc,3*\yc){$\bullet$}; 
  \node (b5) at (1*\xc,4*\yc){$\bullet$};

  \node (c1) at (2*\xc,0*\yc){$\bullet$}; 
  \node (c2) at (2*\xc,1*\yc){$\bullet$}; 
  \node (c3) at (2*\xc,2*\yc){$\bullet$}; 
  \node (c4) at (2*\xc,3*\yc){$\bullet$}; 
  \node (c5) at (2*\xc,4*\yc){$\bullet$}; 
  \node (c6) at (2*\xc,5*\yc){$\bullet$}; 
  \node (c7) at (2*\xc,6*\yc){$\bullet$};

  \node (d1) at (3*\xc,0*\yc){$\bullet$}; 
  \node (d2) at (3*\xc,1*\yc){$\bullet$}; 
  \node (d3) at (3*\xc,2*\yc){$\bullet$}; 
  \node (d4) at (3*\xc,3*\yc){$\bullet$}; 
  \node (d5) at (3*\xc,4*\yc){$\bullet$}; 
  \node (d6) at (3*\xc,5*\yc){$\bullet$}; 
  \node (d7) at (3*\xc,6*\yc){$\bullet$}; 
  \node (d8) at (3*\xc,7*\yc){$\bullet$};

  \node (e1) at (4*\xc,0*\yc){$\bullet$}; 
  \node (e2) at (4*\xc,1*\yc){$\bullet$}; 
  \node (e3) at (4*\xc,2*\yc){$\bullet$}; 
  \node (e4) at (4*\xc,3*\yc){$\bullet$}; 
  \node (e5) at (4*\xc,4*\yc){$\bullet$}; 
  \node (e6) at (4*\xc,5*\yc){$\bullet$}; 
  \node (e7) at (4*\xc,6*\yc){$\bullet$}; 
  \node (e8) at (4*\xc,7*\yc){$\bullet$};

  \node (f1) at (5*\xc,0*\yc){$\bullet$}; 
  \node (f2) at (5*\xc,1*\yc){$\bullet$}; 
  \node (f3) at (5*\xc,2*\yc){$\bullet$}; 
  \node (f4) at (5*\xc,3*\yc){$\bullet$}; 
  \node (f5) at (5*\xc,4*\yc){$\bullet$}; 
  \node (f6) at (5*\xc,5*\yc){$\bullet$}; 
  \node (f7) at (5*\xc,6*\yc){$\bullet$}; 
  \node (f8) at (5*\xc,7*\yc){$\bullet$}; 
  \node (f9) at (5*\xc,8*\yc){$\bullet$}; 

\draw[thick] (b1.center) -- (f1.center);
\draw[thick] (b2.center) -- (f2.center);
\draw[thick] (b3.center) -- (f3.center);

\draw[thick] (a4.center) -- (b4.center); \draw[thick] (c4.center) -- (f4.center);
\draw[thick] (a5.center) -- (d5.center); \draw[thick] (e5.center) -- (f5.center);

\draw[thick] (c6.center) -- (f6.center);
\draw[thick] (c7.center) -- (f7.center);
\draw[thick] (d8.center) -- (f8.center);
\end{tikzpicture}
\]
\caption{The optimal lace diagrams for dimension vector $\ud=(5,5,7,8,8,9)$ in Theorem~\ref{thm:components}. They correspond to the $\theta=4$ solutions of \eqref{eq:QIP}: $(4,1,0,0,0)$ (top left),  $(3,2,0,0,0)$ (top right), $(3,1,1,0,0)$ (bottom left), $(3,1,0,1,0)$ (bottom right). }  \label{fig:557889_original}
\end{figure}
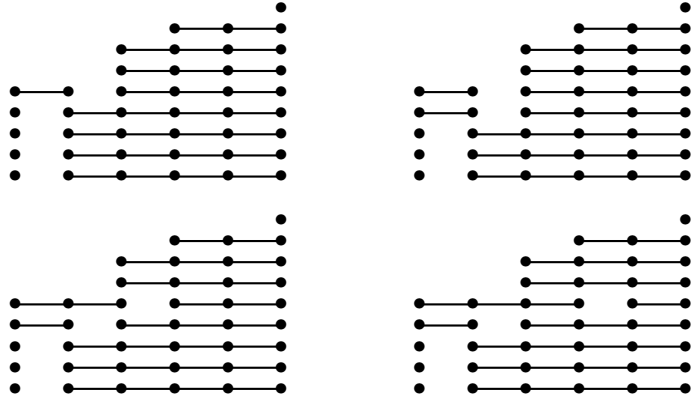

In Figure~\ref{fig:557889_original} we show the outcome for this algorithm for $\ud=(5,5,7,8,8,9)$ and the four ($\theta=4$) optimal solutions of~\eqref{eq:QIP}: $\ue=(4,1,0,0,0)$, $\ue=(3,2,0,0,0)$, $\ue=(3,1,1,0,0)$, $\ue=(3,1,0,1,0)$.

\begin{theorem}\cite{LR}
\label{thm:components}
Consider the lace diagrams obtained by the algorithm above for the optimal solutions $\ue$ of~\eqref{eq:QIP}. They determine $\theta$ Kostant partitions of $\ud$. The orbit closures associated with these Kostant partitions are exactly the $\theta$ maximal-dimensional components of~$\Sigma_{\ud}$. \qed
 \end{theorem}

\section{Description of the maximal-dimensional components of $\Sigma_{\ud}$}
\label{sec:results}

Let us fix a dimension vector $\dv=(d_0,\ldots, d_n)$ and an integer $k\in[0,n]$ such that $d_k=\min(d_0,\ldots, d_n).$

\begin{definition}
	A \rvector \ $\rv$ is a tuple $(e_0,\ldots, e_{k-1}, \star, e_{k+1} \ldots, e_n)$ of nonnegative integers such that
	$\sum_{i\neq k} e_i=d_k$. The $\star$ is used as a placeholder that indicates the value of~$k$.
\end{definition}
To every \rvector \ we  are going to associate a lace diagram $D(\rv)$, and hence a Kostant partition of $\ud$, and in turn a $\GL_{\ud}$-orbit in $\Sigma_{\ud}$.
\begin{definition}
	Let $\rv$ be a \rvector. We associate to it a set of integer points (dots)
	\begin{align*}
	&\ \{(x,y)\in \Z^2| x\in[0,k-1]\,, d_k-d_x-\sum^{k-1}_{i=x} e_i\le y < d_k-\sum^{k-1}_{i=x} e_i\}  \\
	&\cup \{(x,y)\in \Z^2| x=k\,, 0\le y < d_k \}  \\
	&\cup \{(x,y)\in \Z^2| x\in[k+1,n]\,, \sum^{x}_{i=k+1}e_i \le y < d_x +\sum^{x}_{i=k+1}e_i \}.
	\end{align*}
	The lace diagram $\Dset{\rv}$ is obtained by drawing all possible horizontal segments of length 1. We consider the corresponding quiver module $\Vrep{\rv}$, and the $\GL_{\ud}$-orbit $\Orep{\rv}\subset \Rep_{\dv}$.
\end{definition}
\begin{definition}
	For a subset $A\subset \R$ we consider the submodule $\Vrep{\rv}_A\subseteq \Vrep{\rv}$ given by dots whose $y$-coordinate lies in $A$.
\end{definition}
\begin{remark} \rm
	The set of dots in the lace diagram $\Dset{\rv}$ can be described more intuitively as follows. For each $x$ ranging from $0$ to $n$, place $d_x$ consecutive dots in column $x$.
	\begin{itemize}
		\item In column $k$, the dots occupy positions from height $0$ to $d_k - 1$.
		\item For columns to the right of $k$, align the bottoms of adjacent columns so that the bottom of column $i$ is $e_i$ units higher than the bottom of column $i-1$.
		\item For columns to the left of $k$, align the tops of adjacent columns so that the top of column $i$ is $e_i$ units lower than the top of column $i+1$.
	\end{itemize} 
\end{remark}

\begin{figure}
\[
\begin{tikzpicture}[scale=\magn]
  \node[blue] (a1) at (0*\xc,0*\yc){$\bullet$}; 
  \node[blue] (a2) at (0*\xc,1*\yc){$\bullet$}; 
  \node[blue] (a3) at (0*\xc,2*\yc){$\bullet$}; 
  \node[blue] (a4) at (0*\xc,3*\yc){$\bullet$}; 
  \node[blue] (a5) at (0*\xc,4*\yc){$\bullet$};

  \node (b1) at (1*\xc,4*\yc){$\bullet$}; 
  \node (b2) at (1*\xc,5*\yc){$\bullet$}; 
  \node (b3) at (1*\xc,6*\yc){$\bullet$}; 
  \node (b4) at (1*\xc,7*\yc){$\bullet$}; 
  \node (b5) at (1*\xc,8*\yc){$\bullet$};

  \node (c1) at (2*\xc,5*\yc){$\bullet$}; 
  \node (c2) at (2*\xc,6*\yc){$\bullet$}; 
  \node (c3) at (2*\xc,7*\yc){$\bullet$}; 
  \node (c4) at (2*\xc,8*\yc){$\bullet$}; 
  \node (c5) at (2*\xc,9*\yc){$\bullet$}; 
  \node (c6) at (2*\xc,10*\yc){$\bullet$}; 
  \node (c7) at (2*\xc,11*\yc){$\bullet$};

  \node (d1) at (3*\xc,5*\yc){$\bullet$}; 
  \node (d2) at (3*\xc,6*\yc){$\bullet$}; 
  \node (d3) at (3*\xc,7*\yc){$\bullet$}; 
  \node (d4) at (3*\xc,8*\yc){$\bullet$}; 
  \node (d5) at (3*\xc,9*\yc){$\bullet$}; 
  \node (d6) at (3*\xc,10*\yc){$\bullet$}; 
  \node (d7) at (3*\xc,11*\yc){$\bullet$}; 
  \node (d8) at (3*\xc,12*\yc){$\bullet$};

  \node (e1) at (4*\xc,5*\yc){$\bullet$}; 
  \node (e2) at (4*\xc,6*\yc){$\bullet$}; 
  \node (e3) at (4*\xc,7*\yc){$\bullet$}; 
  \node (e4) at (4*\xc,8*\yc){$\bullet$}; 
  \node (e5) at (4*\xc,9*\yc){$\bullet$}; 
  \node (e6) at (4*\xc,10*\yc){$\bullet$}; 
  \node (e7) at (4*\xc,11*\yc){$\bullet$}; 
  \node (e8) at (4*\xc,12*\yc){$\bullet$};

  \node (f1) at (5*\xc,5*\yc){$\bullet$}; 
  \node (f2) at (5*\xc,6*\yc){$\bullet$}; 
  \node (f3) at (5*\xc,7*\yc){$\bullet$}; 
  \node (f4) at (5*\xc,8*\yc){$\bullet$}; 
  \node (f5) at (5*\xc,9*\yc){$\bullet$}; 
  \node (f6) at (5*\xc,10*\yc){$\bullet$}; 
  \node (f7) at (5*\xc,11*\yc){$\bullet$}; 
  \node (f8) at (5*\xc,12*\yc){$\bullet$}; 
  \node (f9) at (5*\xc,13*\yc){$\bullet$};

\draw[thick] (a5.center) -- (b1.center); 
\draw[thick] (b2.center) -- (f1.center); 
\draw[thick] (b3.center) -- (f2.center);
\draw[thick] (b4.center) -- (f3.center);
\draw[thick] (b5.center) -- (f4.center);
\draw[thick] (c5.center) -- (f5.center);
\draw[thick] (c6.center) -- (f6.center);
\draw[thick] (c7.center) -- (f7.center);
\draw[thick] (d8.center) -- (f8.center);

\draw[->,red]   (\xc, 0) --   (\xc,3*\yc) node[midway,right] {$4$};
\draw[->,red] (2*\xc, 3*\yc) -- (2*\xc,4*\yc) node[midway,right] {$1$};
\draw[->,red] (3*\xc, 3*\yc) -- (3*\xc,4*\yc) node[midway,right] {$0$};
\draw[->,red] (4*\xc, 3*\yc) -- (4*\xc,4*\yc) node[midway,right] {$0$};
\draw[->,red] (5*\xc, 3*\yc) -- (5*\xc,4*\yc) node[midway,right] {$0$};

\end{tikzpicture}
\qquad\qquad
\begin{tikzpicture}[scale=\magn]
  \node[blue] (a1) at (0*\xc,0*\yc){$\bullet$}; 
  \node[blue] (a2) at (0*\xc,1*\yc){$\bullet$}; 
  \node[blue] (a3) at (0*\xc,2*\yc){$\bullet$}; 
  \node[blue] (a4) at (0*\xc,3*\yc){$\bullet$}; 
  \node[blue] (a5) at (0*\xc,4*\yc){$\bullet$};

  \node (b1) at (1*\xc,3*\yc){$\bullet$}; 
  \node (b2) at (1*\xc,4*\yc){$\bullet$}; 
  \node (b3) at (1*\xc,5*\yc){$\bullet$}; 
  \node (b4) at (1*\xc,6*\yc){$\bullet$}; 
  \node (b5) at (1*\xc,7*\yc){$\bullet$};

  \node (c1) at (2*\xc,5*\yc){$\bullet$}; 
  \node (c2) at (2*\xc,6*\yc){$\bullet$}; 
  \node (c3) at (2*\xc,7*\yc){$\bullet$}; 
  \node (c4) at (2*\xc,8*\yc){$\bullet$}; 
  \node (c5) at (2*\xc,9*\yc){$\bullet$}; 
  \node (c6) at (2*\xc,10*\yc){$\bullet$}; 
  \node (c7) at (2*\xc,11*\yc){$\bullet$};

  \node (d1) at (3*\xc,5*\yc){$\bullet$}; 
  \node (d2) at (3*\xc,6*\yc){$\bullet$}; 
  \node (d3) at (3*\xc,7*\yc){$\bullet$}; 
  \node (d4) at (3*\xc,8*\yc){$\bullet$}; 
  \node (d5) at (3*\xc,9*\yc){$\bullet$}; 
  \node (d6) at (3*\xc,10*\yc){$\bullet$}; 
  \node (d7) at (3*\xc,11*\yc){$\bullet$}; 
  \node (d8) at (3*\xc,12*\yc){$\bullet$};

  \node (e1) at (4*\xc,5*\yc){$\bullet$}; 
  \node (e2) at (4*\xc,6*\yc){$\bullet$}; 
  \node (e3) at (4*\xc,7*\yc){$\bullet$}; 
  \node (e4) at (4*\xc,8*\yc){$\bullet$}; 
  \node (e5) at (4*\xc,9*\yc){$\bullet$}; 
  \node (e6) at (4*\xc,10*\yc){$\bullet$}; 
  \node (e7) at (4*\xc,11*\yc){$\bullet$}; 
  \node (e8) at (4*\xc,12*\yc){$\bullet$};

  \node (f1) at (5*\xc,5*\yc){$\bullet$}; 
  \node (f2) at (5*\xc,6*\yc){$\bullet$}; 
  \node (f3) at (5*\xc,7*\yc){$\bullet$}; 
  \node (f4) at (5*\xc,8*\yc){$\bullet$}; 
  \node (f5) at (5*\xc,9*\yc){$\bullet$}; 
  \node (f6) at (5*\xc,10*\yc){$\bullet$}; 
  \node (f7) at (5*\xc,11*\yc){$\bullet$}; 
  \node (f8) at (5*\xc,12*\yc){$\bullet$}; 
  \node (f9) at (5*\xc,13*\yc){$\bullet$};

\draw[thick] (a4.center) -- (b1.center); 
\draw[thick] (a5.center) -- (b2.center); 
\draw[thick] (b3.center) -- (f1.center);
\draw[thick] (b4.center) -- (f2.center);
\draw[thick] (b5.center) -- (f3.center);
\draw[thick] (c4.center) -- (f4.center);
\draw[thick] (c5.center) -- (f5.center);
\draw[thick] (c6.center) -- (f6.center);
\draw[thick] (c7.center) -- (f7.center);
\draw[thick] (d8.center) -- (f8.center);

\draw[->,red]   (\xc, 0) --   (\xc,2*\yc) node[midway,right] {$3$};
\draw[->,red] (2*\xc, 2.5*\yc) -- (2*\xc,4*\yc) node[midway,right] {$2$};
\draw[->,red] (3*\xc, 2.5*\yc) -- (3*\xc,4*\yc) node[midway,right] {$0$};
\draw[->,red] (4*\xc, 2.5*\yc) -- (4*\xc,4*\yc) node[midway,right] {$0$};
\draw[->,red] (5*\xc, 2.5*\yc) -- (5*\xc,4*\yc) node[midway,right] {$0$};

\end{tikzpicture}
\]
\[
\begin{tikzpicture}[scale=\magn]
  \node[blue] (a1) at (0*\xc,0*\yc){$\bullet$}; 
  \node[blue] (a2) at (0*\xc,1*\yc){$\bullet$}; 
  \node[blue] (a3) at (0*\xc,2*\yc){$\bullet$}; 
  \node[blue] (a4) at (0*\xc,3*\yc){$\bullet$}; 
  \node[blue] (a5) at (0*\xc,4*\yc){$\bullet$};

  \node (b1) at (1*\xc,3*\yc){$\bullet$}; 
  \node (b2) at (1*\xc,4*\yc){$\bullet$}; 
  \node (b3) at (1*\xc,5*\yc){$\bullet$}; 
  \node (b4) at (1*\xc,6*\yc){$\bullet$}; 
  \node (b5) at (1*\xc,7*\yc){$\bullet$};

  \node (c1) at (2*\xc,4*\yc){$\bullet$}; 
  \node (c2) at (2*\xc,5*\yc){$\bullet$}; 
  \node (c3) at (2*\xc,6*\yc){$\bullet$}; 
  \node (c4) at (2*\xc,7*\yc){$\bullet$}; 
  \node (c5) at (2*\xc,8*\yc){$\bullet$}; 
  \node (c6) at (2*\xc,9*\yc){$\bullet$}; 
  \node (c7) at (2*\xc,10*\yc){$\bullet$};

  \node (d1) at (3*\xc,5*\yc){$\bullet$}; 
  \node (d2) at (3*\xc,6*\yc){$\bullet$}; 
  \node (d3) at (3*\xc,7*\yc){$\bullet$}; 
  \node (d4) at (3*\xc,8*\yc){$\bullet$}; 
  \node (d5) at (3*\xc,9*\yc){$\bullet$}; 
  \node (d6) at (3*\xc,10*\yc){$\bullet$}; 
  \node (d7) at (3*\xc,11*\yc){$\bullet$}; 
  \node (d8) at (3*\xc,12*\yc){$\bullet$};

  \node (e1) at (4*\xc,5*\yc){$\bullet$}; 
  \node (e2) at (4*\xc,6*\yc){$\bullet$}; 
  \node (e3) at (4*\xc,7*\yc){$\bullet$}; 
  \node (e4) at (4*\xc,8*\yc){$\bullet$}; 
  \node (e5) at (4*\xc,9*\yc){$\bullet$}; 
  \node (e6) at (4*\xc,10*\yc){$\bullet$}; 
  \node (e7) at (4*\xc,11*\yc){$\bullet$}; 
  \node (e8) at (4*\xc,12*\yc){$\bullet$};

  \node (f1) at (5*\xc,5*\yc){$\bullet$}; 
  \node (f2) at (5*\xc,6*\yc){$\bullet$}; 
  \node (f3) at (5*\xc,7*\yc){$\bullet$}; 
  \node (f4) at (5*\xc,8*\yc){$\bullet$}; 
  \node (f5) at (5*\xc,9*\yc){$\bullet$}; 
  \node (f6) at (5*\xc,10*\yc){$\bullet$}; 
  \node (f7) at (5*\xc,11*\yc){$\bullet$}; 
  \node (f8) at (5*\xc,12*\yc){$\bullet$}; 
  \node (f9) at (5*\xc,13*\yc){$\bullet$};

\draw[thick] (a4.center) -- (b1.center); 
\draw[thick] (a5.center) -- (c1.center); 
\draw[thick] (b3.center) -- (f1.center);
\draw[thick] (b4.center) -- (f2.center);
\draw[thick] (b5.center) -- (f3.center);
\draw[thick] (c5.center) -- (f4.center);
\draw[thick] (c6.center) -- (f5.center);
\draw[thick] (c7.center) -- (f6.center);
\draw[thick] (d7.center) -- (f7.center);
\draw[thick] (d8.center) -- (f8.center);

\draw[->,red]   (\xc, 0) --   (\xc,2*\yc) node[midway,right] {$3$};
\draw[->,red] (2*\xc, 2*\yc) -- (2*\xc,3*\yc) node[midway,right] {$1$};
\draw[->,red] (3*\xc, 3*\yc) -- (3*\xc,4*\yc) node[midway,right] {$1$};
\draw[->,red] (4*\xc, 3*\yc) -- (4*\xc,4*\yc) node[midway,right] {$0$};
\draw[->,red] (5*\xc, 3*\yc) -- (5*\xc,4*\yc) node[midway,right] {$0$};

\end{tikzpicture}
\qquad\qquad
\begin{tikzpicture}[scale=\magn]
  \node[blue] (a1) at (0*\xc,0*\yc){$\bullet$}; 
  \node[blue] (a2) at (0*\xc,1*\yc){$\bullet$}; 
  \node[blue] (a3) at (0*\xc,2*\yc){$\bullet$}; 
  \node[blue] (a4) at (0*\xc,3*\yc){$\bullet$}; 
  \node[blue] (a5) at (0*\xc,4*\yc){$\bullet$};

  \node (b1) at (1*\xc,3*\yc){$\bullet$}; 
  \node (b2) at (1*\xc,4*\yc){$\bullet$}; 
  \node (b3) at (1*\xc,5*\yc){$\bullet$}; 
  \node (b4) at (1*\xc,6*\yc){$\bullet$}; 
  \node (b5) at (1*\xc,7*\yc){$\bullet$};

  \node (c1) at (2*\xc,4*\yc){$\bullet$}; 
  \node (c2) at (2*\xc,5*\yc){$\bullet$}; 
  \node (c3) at (2*\xc,6*\yc){$\bullet$}; 
  \node (c4) at (2*\xc,7*\yc){$\bullet$}; 
  \node (c5) at (2*\xc,8*\yc){$\bullet$}; 
  \node (c6) at (2*\xc,9*\yc){$\bullet$}; 
  \node (c7) at (2*\xc,10*\yc){$\bullet$};

  \node (d1) at (3*\xc,4*\yc){$\bullet$}; 
  \node (d2) at (3*\xc,5*\yc){$\bullet$}; 
  \node (d3) at (3*\xc,6*\yc){$\bullet$}; 
  \node (d4) at (3*\xc,7*\yc){$\bullet$}; 
  \node (d5) at (3*\xc,8*\yc){$\bullet$}; 
  \node (d6) at (3*\xc,9*\yc){$\bullet$}; 
  \node (d7) at (3*\xc,10*\yc){$\bullet$}; 
  \node (d8) at (3*\xc,11*\yc){$\bullet$};

  \node (e1) at (4*\xc,5*\yc){$\bullet$}; 
  \node (e2) at (4*\xc,6*\yc){$\bullet$}; 
  \node (e3) at (4*\xc,7*\yc){$\bullet$}; 
  \node (e4) at (4*\xc,8*\yc){$\bullet$}; 
  \node (e5) at (4*\xc,9*\yc){$\bullet$}; 
  \node (e6) at (4*\xc,10*\yc){$\bullet$}; 
  \node (e7) at (4*\xc,11*\yc){$\bullet$}; 
  \node (e8) at (4*\xc,12*\yc){$\bullet$};

  \node (f1) at (5*\xc,5*\yc){$\bullet$}; 
  \node (f2) at (5*\xc,6*\yc){$\bullet$}; 
  \node (f3) at (5*\xc,7*\yc){$\bullet$}; 
  \node (f4) at (5*\xc,8*\yc){$\bullet$}; 
  \node (f5) at (5*\xc,9*\yc){$\bullet$}; 
  \node (f6) at (5*\xc,10*\yc){$\bullet$}; 
  \node (f7) at (5*\xc,11*\yc){$\bullet$}; 
  \node (f8) at (5*\xc,12*\yc){$\bullet$}; 
  \node (f9) at (5*\xc,13*\yc){$\bullet$};

\draw[thick] (a4.center) -- (b1.center); 
\draw[thick] (a5.center) -- (d1.center); 
\draw[thick] (b3.center) -- (f1.center);
\draw[thick] (b4.center) -- (f2.center);
\draw[thick] (b5.center) -- (f3.center);
\draw[thick] (c5.center) -- (f4.center);
\draw[thick] (c6.center) -- (f5.center);
\draw[thick] (c7.center) -- (f6.center);
\draw[thick] (d8.center) -- (f7.center);
\draw[thick] (e8.center) -- (f8.center);

\draw[->,red]   (\xc, 0) --   (\xc,2*\yc) node[midway,right] {$3$};
\draw[->,red] (2*\xc, 2*\yc) -- (2*\xc,3*\yc) node[midway,right] {$1$};
\draw[->,red] (3*\xc, 2*\yc) -- (3*\xc,3*\yc) node[midway,right] {$0$};
\draw[->,red] (4*\xc, 3*\yc) -- (4*\xc,4*\yc) node[midway,right] {$1$};
\draw[->,red] (5*\xc, 3*\yc) -- (5*\xc,4*\yc) node[midway,right] {$0$};

\end{tikzpicture}
\]
\caption{The lace diagrams associated to the dimension vector $\ud=(5,5,7,8,8,9)$, and raising vectors $(\star,4,1,0,0,0)$ (top left),  $(\star,3,2,0,0,0)$ (top right), $(\star,3,1,1,0,0)$ (bottom left), $(\star,3,1,0,1,0)$ (bottom right). Notice that these lace diagrams encode the {\em same} Kostant partitions as the ones in Figure~\ref{fig:557889_original}. Also, these raising vectors are the solutions of (QIP'), hence these lace diagrams describe the maximal-dimensional components of~$\Sigma_{\ud}$.}  
\label{fig:557889 raising vectors}
\end{figure}
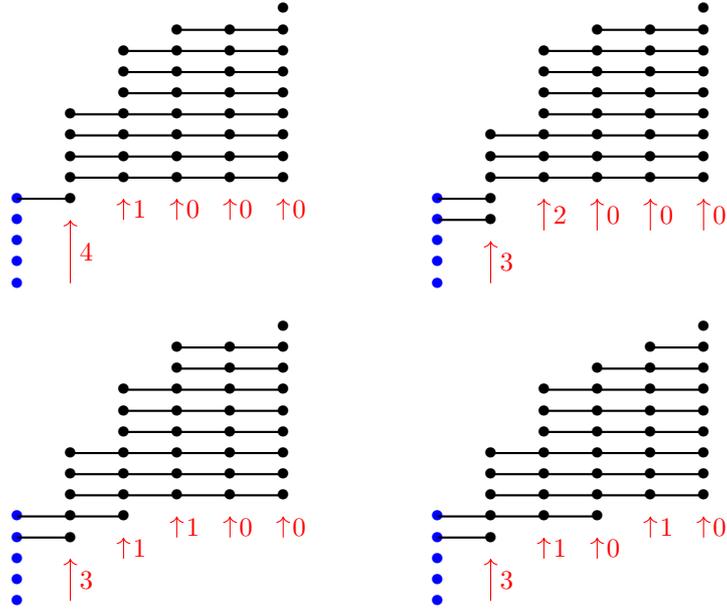

\begin{example} \rm
    Let $\ud=(5,5,7,8,8,9)$. The lace diagrams for $k=0$ and the raising vectors $(\star,4,1,0,0,0)$, $(\star,3,2,0,0,0)$, $(\star,3,1,1,0,0)$, $(\star,3,1,0,1,0)$ are shown in Figure~\ref{fig:557889 raising vectors}.
\end{example}

\begin{figure}
\[
\begin{tikzpicture}[scale=\magn]
  \node (a1) at (0*\xc,0*\yc){$\bullet$}; 
  \node (a2) at (0*\xc,1*\yc){$\bullet$}; 
  \node (a3) at (0*\xc,2*\yc){$\bullet$}; 
  \node (a4) at (0*\xc,3*\yc){$\bullet$}; 
  \node (a5) at (0*\xc,4*\yc){$\bullet$}; 
  \node (a6) at (0*\xc,5*\yc){$\bullet$}; 
  \node (a7) at (0*\xc,6*\yc){$\bullet$}; 
  \node (a8) at (0*\xc,7*\yc){$\bullet$}; 

  \node (b1) at (1*\xc,1*\yc){$\bullet$}; 
  \node (b2) at (1*\xc,2*\yc){$\bullet$}; 
  \node (b3) at (1*\xc,3*\yc){$\bullet$}; 
  \node (b4) at (1*\xc,4*\yc){$\bullet$}; 
  \node (b5) at (1*\xc,5*\yc){$\bullet$}; 
  \node (b6) at (1*\xc,6*\yc){$\bullet$};
  \node (b7) at (1*\xc,7*\yc){$\bullet$};

  \node (c1) at (2*\xc,4*\yc){$\bullet$}; 
  \node (c2) at (2*\xc,5*\yc){$\bullet$}; 
  \node (c3) at (2*\xc,6*\yc){$\bullet$}; 
  \node (c4) at (2*\xc,7*\yc){$\bullet$}; 
  \node (c5) at (2*\xc,8*\yc){$\bullet$}; 

  \node (d1) at (3*\xc,4*\yc){$\bullet$}; 
  \node (d2) at (3*\xc,5*\yc){$\bullet$}; 
  \node (d3) at (3*\xc,6*\yc){$\bullet$}; 
  \node (d4) at (3*\xc,7*\yc){$\bullet$}; 
  \node (d5) at (3*\xc,8*\yc){$\bullet$}; 
  \node (d6) at (3*\xc,9*\yc){$\bullet$}; 
  \node (d7) at (3*\xc,10*\yc){$\bullet$}; 
  \node (d8) at (3*\xc,11*\yc){$\bullet$}; 
  \node (d9) at (3*\xc,12*\yc){$\bullet$}; 

  \node (e1) at (4*\xc,8*\yc){$\bullet$}; 
  \node (e2) at (4*\xc,9*\yc){$\bullet$}; 
  \node (e3) at (4*\xc,10*\yc){$\bullet$}; 
  \node (e4) at (4*\xc,11*\yc){$\bullet$}; 
  \node (e5) at (4*\xc,12*\yc){$\bullet$}; 

  \node (f1) at (5*\xc,8*\yc){$\bullet$}; 
  \node (f2) at (5*\xc,9*\yc){$\bullet$}; 
  \node (f3) at (5*\xc,10*\yc){$\bullet$}; 
  \node (f4) at (5*\xc,11*\yc){$\bullet$}; 
  \node (f5) at (5*\xc,12*\yc){$\bullet$}; 
  \node (f6) at (5*\xc,13*\yc){$\bullet$}; 
  \node (f7) at (5*\xc,14*\yc){$\bullet$}; 
  \node (f8) at (5*\xc,15*\yc){$\bullet$};

\draw[thick] (a2.center) -- (b1.center); 
\draw[thick] (a3.center) -- (b2.center); 
\draw[thick] (a4.center) -- (b3.center); 
\draw[thick] (a5.center) -- (d1.center); 
\draw[thick] (a6.center) -- (d2.center); 
\draw[thick] (a7.center) -- (d3.center); 
\draw[thick] (a8.center) -- (d4.center); 
\draw[thick] (c4.center) -- (d4.center); 
\draw[thick] (c5.center) -- (f1.center); 
\draw[thick] (d6.center) -- (f2.center); 
\draw[thick] (d7.center) -- (f3.center); 
\draw[thick] (d8.center) -- (f4.center); 
\draw[thick] (d9.center) -- (f5.center);

\draw[->,red]   (0*\xc, 9*\yc) --  (0*\xc,8*\yc) node[midway,right] {$0$};
\draw[->,red]   (1*\xc, 9*\yc) --  (1*\xc,8*\yc) node[midway,right] {$1$};
\draw[->,red]   (3*\xc, 2*\yc) --  (3*\xc,3*\yc) node[midway,right] {$0$};
\draw[->,red]   (4*\xc, 4*\yc) --  (4*\xc,7*\yc) node[midway,right] {$4$};
\draw[->,red]   (5*\xc, 6*\yc) --  (5*\xc,7*\yc) node[midway,right] {$0$};
  \node[blue] (c1) at (2*\xc,4*\yc){$\bullet$}; 
  \node[blue] (c2) at (2*\xc,5*\yc){$\bullet$}; 
  \node[blue] (c3) at (2*\xc,6*\yc){$\bullet$}; 
  \node[blue] (c4) at (2*\xc,7*\yc){$\bullet$}; 
  \node[blue] (c5) at (2*\xc,8*\yc){$\bullet$}; 
\end{tikzpicture}
\qquad\qquad
\begin{tikzpicture}[scale=\magn]
  \node (a1) at (0*\xc,0*\yc){$\bullet$}; 
  \node (a2) at (0*\xc,1*\yc){$\bullet$}; 
  \node (a3) at (0*\xc,2*\yc){$\bullet$}; 
  \node (a4) at (0*\xc,3*\yc){$\bullet$}; 
  \node (a5) at (0*\xc,4*\yc){$\bullet$}; 
  \node (a6) at (0*\xc,5*\yc){$\bullet$}; 
  \node (a7) at (0*\xc,6*\yc){$\bullet$}; 
  \node (a8) at (0*\xc,7*\yc){$\bullet$}; 

  \node (b1) at (1*\xc,1*\yc){$\bullet$}; 
  \node (b2) at (1*\xc,2*\yc){$\bullet$}; 
  \node (b3) at (1*\xc,3*\yc){$\bullet$}; 
  \node (b4) at (1*\xc,4*\yc){$\bullet$}; 
  \node (b5) at (1*\xc,5*\yc){$\bullet$}; 
  \node (b6) at (1*\xc,6*\yc){$\bullet$};
  \node (b7) at (1*\xc,7*\yc){$\bullet$};

  \node (c1) at (2*\xc,5*\yc){$\bullet$}; 
  \node (c2) at (2*\xc,6*\yc){$\bullet$}; 
  \node (c3) at (2*\xc,7*\yc){$\bullet$}; 
  \node (c4) at (2*\xc,8*\yc){$\bullet$}; 
  \node (c5) at (2*\xc,9*\yc){$\bullet$}; 

  \node (d1) at (3*\xc,5*\yc){$\bullet$}; 
  \node (d2) at (3*\xc,6*\yc){$\bullet$}; 
  \node (d3) at (3*\xc,7*\yc){$\bullet$}; 
  \node (d4) at (3*\xc,8*\yc){$\bullet$}; 
  \node (d5) at (3*\xc,9*\yc){$\bullet$}; 
  \node (d6) at (3*\xc,10*\yc){$\bullet$}; 
  \node (d7) at (3*\xc,11*\yc){$\bullet$}; 
  \node (d8) at (3*\xc,12*\yc){$\bullet$}; 
  \node (d9) at (3*\xc,13*\yc){$\bullet$}; 

  \node (e1) at (4*\xc,8*\yc){$\bullet$}; 
  \node (e2) at (4*\xc,9*\yc){$\bullet$}; 
  \node (e3) at (4*\xc,10*\yc){$\bullet$}; 
  \node (e4) at (4*\xc,11*\yc){$\bullet$}; 
  \node (e5) at (4*\xc,12*\yc){$\bullet$}; 

  \node (f1) at (5*\xc,8*\yc){$\bullet$}; 
  \node (f2) at (5*\xc,9*\yc){$\bullet$}; 
  \node (f3) at (5*\xc,10*\yc){$\bullet$}; 
  \node (f4) at (5*\xc,11*\yc){$\bullet$}; 
  \node (f5) at (5*\xc,12*\yc){$\bullet$}; 
  \node (f6) at (5*\xc,13*\yc){$\bullet$}; 
  \node (f7) at (5*\xc,14*\yc){$\bullet$}; 
  \node (f8) at (5*\xc,15*\yc){$\bullet$};

\draw[thick] (a2.center) -- (b1.center); 
\draw[thick] (a3.center) -- (b2.center); 
\draw[thick] (a4.center) -- (b3.center); 
\draw[thick] (a5.center) -- (b4.center); 
\draw[thick] (a6.center) -- (d1.center); 
\draw[thick] (a7.center) -- (d2.center); 
\draw[thick] (a8.center) -- (d3.center); 
\draw[thick] (c4.center) -- (f1.center); 
\draw[thick] (c5.center) -- (f2.center); 
\draw[thick] (d6.center) -- (f3.center); 
\draw[thick] (d7.center) -- (f4.center); 
\draw[thick] (d8.center) -- (f5.center);

\draw[->,red]   (0*\xc, 9*\yc) --  (0*\xc,8*\yc) node[midway,right] {$0$};
\draw[->,red]   (1*\xc, 9*\yc) --  (1*\xc,8*\yc) node[midway,right] {$2$};
\draw[->,red]   (3*\xc, 3*\yc) --  (3*\xc,4*\yc) node[midway,right] {$0$};
\draw[->,red]   (4*\xc, 5*\yc) --  (4*\xc,7*\yc) node[midway,right] {$3$};
\draw[->,red]   (5*\xc, 6*\yc) --  (5*\xc,7*\yc) node[midway,right] {$0$};
  \node[blue] (c1) at (2*\xc,5*\yc){$\bullet$}; 
  \node[blue]  (c2) at (2*\xc,6*\yc){$\bullet$}; 
  \node[blue]  (c3) at (2*\xc,7*\yc){$\bullet$}; 
  \node[blue]  (c4) at (2*\xc,8*\yc){$\bullet$}; 
  \node[blue]  (c5) at (2*\xc,9*\yc){$\bullet$}; 
\end{tikzpicture}
\]
\[
\begin{tikzpicture}[scale=\magn]
  \node (a1) at (0*\xc,0*\yc){$\bullet$}; 
  \node (a2) at (0*\xc,1*\yc){$\bullet$}; 
  \node (a3) at (0*\xc,2*\yc){$\bullet$}; 
  \node (a4) at (0*\xc,3*\yc){$\bullet$}; 
  \node (a5) at (0*\xc,4*\yc){$\bullet$}; 
  \node (a6) at (0*\xc,5*\yc){$\bullet$}; 
  \node (a7) at (0*\xc,6*\yc){$\bullet$}; 
  \node (a8) at (0*\xc,7*\yc){$\bullet$}; 

  \node (b1) at (1*\xc,1*\yc){$\bullet$}; 
  \node (b2) at (1*\xc,2*\yc){$\bullet$}; 
  \node (b3) at (1*\xc,3*\yc){$\bullet$}; 
  \node (b4) at (1*\xc,4*\yc){$\bullet$}; 
  \node (b5) at (1*\xc,5*\yc){$\bullet$}; 
  \node (b6) at (1*\xc,6*\yc){$\bullet$};
  \node (b7) at (1*\xc,7*\yc){$\bullet$};

  \node (c1) at (2*\xc,4*\yc){$\bullet$}; 
  \node (c2) at (2*\xc,5*\yc){$\bullet$}; 
  \node (c3) at (2*\xc,6*\yc){$\bullet$}; 
  \node (c4) at (2*\xc,7*\yc){$\bullet$}; 
  \node (c5) at (2*\xc,8*\yc){$\bullet$}; 

  \node (d1) at (3*\xc,4*\yc){$\bullet$}; 
  \node (d2) at (3*\xc,5*\yc){$\bullet$}; 
  \node (d3) at (3*\xc,6*\yc){$\bullet$}; 
  \node (d4) at (3*\xc,7*\yc){$\bullet$}; 
  \node (d5) at (3*\xc,8*\yc){$\bullet$}; 
  \node (d6) at (3*\xc,9*\yc){$\bullet$}; 
  \node (d7) at (3*\xc,10*\yc){$\bullet$}; 
  \node (d8) at (3*\xc,11*\yc){$\bullet$}; 
  \node (d9) at (3*\xc,12*\yc){$\bullet$}; 

  \node (e1) at (4*\xc,7*\yc){$\bullet$}; 
  \node (e2) at (4*\xc,8*\yc){$\bullet$}; 
  \node (e3) at (4*\xc,9*\yc){$\bullet$}; 
  \node (e4) at (4*\xc,10*\yc){$\bullet$}; 
  \node (e5) at (4*\xc,11*\yc){$\bullet$}; 

  \node (f1) at (5*\xc,8*\yc){$\bullet$}; 
  \node (f2) at (5*\xc,9*\yc){$\bullet$}; 
  \node (f3) at (5*\xc,10*\yc){$\bullet$}; 
  \node (f4) at (5*\xc,11*\yc){$\bullet$}; 
  \node (f5) at (5*\xc,12*\yc){$\bullet$}; 
  \node (f6) at (5*\xc,13*\yc){$\bullet$}; 
  \node (f7) at (5*\xc,14*\yc){$\bullet$}; 
  \node (f8) at (5*\xc,15*\yc){$\bullet$};

\draw[thick] (a2.center) -- (b1.center); 
\draw[thick] (a3.center) -- (b2.center); 
\draw[thick] (a4.center) -- (b3.center); 
\draw[thick] (a5.center) -- (d1.center); 
\draw[thick] (a6.center) -- (d2.center); 
\draw[thick] (a7.center) -- (d3.center); 
\draw[thick] (a8.center) -- (e1.center); 
\draw[thick] (c5.center) -- (f1.center); 
\draw[thick] (d6.center) -- (f2.center); 
\draw[thick] (d7.center) -- (f3.center); 
\draw[thick] (d8.center) -- (f4.center);

\draw[->,red]   (0*\xc, 9*\yc) --  (0*\xc,8*\yc) node[midway,right] {$0$};
\draw[->,red]   (1*\xc, 9*\yc) --  (1*\xc,8*\yc) node[midway,right] {$1$};
\draw[->,red]   (3*\xc, 2*\yc) --  (3*\xc,3*\yc) node[midway,right] {$0$};
\draw[->,red]   (4*\xc, 4*\yc) --  (4*\xc,6*\yc) node[midway,right] {$3$};
\draw[->,red]   (5*\xc, 6*\yc) --  (5*\xc,7*\yc) node[midway,right] {$1$};
  \node[blue] (c1) at (2*\xc,4*\yc){$\bullet$}; 
  \node[blue] (c2) at (2*\xc,5*\yc){$\bullet$}; 
  \node[blue] (c3) at (2*\xc,6*\yc){$\bullet$}; 
  \node[blue] (c4) at (2*\xc,7*\yc){$\bullet$}; 
  \node[blue] (c5) at (2*\xc,8*\yc){$\bullet$}; 
\end{tikzpicture}
\qquad\qquad
\begin{tikzpicture}[scale=\magn]
  \node (a1) at (0*\xc,0*\yc){$\bullet$}; 
  \node (a2) at (0*\xc,1*\yc){$\bullet$}; 
  \node (a3) at (0*\xc,2*\yc){$\bullet$}; 
  \node (a4) at (0*\xc,3*\yc){$\bullet$}; 
  \node (a5) at (0*\xc,4*\yc){$\bullet$}; 
  \node (a6) at (0*\xc,5*\yc){$\bullet$}; 
  \node (a7) at (0*\xc,6*\yc){$\bullet$}; 
  \node (a8) at (0*\xc,7*\yc){$\bullet$}; 

  \node (b1) at (1*\xc,2*\yc){$\bullet$}; 
  \node (b2) at (1*\xc,3*\yc){$\bullet$}; 
  \node (b3) at (1*\xc,4*\yc){$\bullet$}; 
  \node (b4) at (1*\xc,5*\yc){$\bullet$}; 
  \node (b5) at (1*\xc,6*\yc){$\bullet$}; 
  \node (b6) at (1*\xc,7*\yc){$\bullet$};
  \node (b7) at (1*\xc,8*\yc){$\bullet$};

  \node (c1) at (2*\xc,5*\yc){$\bullet$}; 
  \node (c2) at (2*\xc,6*\yc){$\bullet$}; 
  \node (c3) at (2*\xc,7*\yc){$\bullet$}; 
  \node (c4) at (2*\xc,8*\yc){$\bullet$}; 
  \node (c5) at (2*\xc,9*\yc){$\bullet$}; 

  \node (d1) at (3*\xc,5*\yc){$\bullet$}; 
  \node (d2) at (3*\xc,6*\yc){$\bullet$}; 
  \node (d3) at (3*\xc,7*\yc){$\bullet$}; 
  \node (d4) at (3*\xc,8*\yc){$\bullet$}; 
  \node (d5) at (3*\xc,9*\yc){$\bullet$}; 
  \node (d6) at (3*\xc,10*\yc){$\bullet$}; 
  \node (d7) at (3*\xc,11*\yc){$\bullet$}; 
  \node (d8) at (3*\xc,12*\yc){$\bullet$}; 
  \node (d9) at (3*\xc,13*\yc){$\bullet$}; 

  \node (e1) at (4*\xc,8*\yc){$\bullet$}; 
  \node (e2) at (4*\xc,9*\yc){$\bullet$}; 
  \node (e3) at (4*\xc,10*\yc){$\bullet$}; 
  \node (e4) at (4*\xc,11*\yc){$\bullet$}; 
  \node (e5) at (4*\xc,12*\yc){$\bullet$}; 

  \node (f1) at (5*\xc,8*\yc){$\bullet$}; 
  \node (f2) at (5*\xc,9*\yc){$\bullet$}; 
  \node (f3) at (5*\xc,10*\yc){$\bullet$}; 
  \node (f4) at (5*\xc,11*\yc){$\bullet$}; 
  \node (f5) at (5*\xc,12*\yc){$\bullet$}; 
  \node (f6) at (5*\xc,13*\yc){$\bullet$}; 
  \node (f7) at (5*\xc,14*\yc){$\bullet$}; 
  \node (f8) at (5*\xc,15*\yc){$\bullet$};

\draw[thick] (a3.center) -- (b1.center); 
\draw[thick] (a4.center) -- (b2.center); 
\draw[thick] (a5.center) -- (b3.center); 
\draw[thick] (a6.center) -- (d1.center); 
\draw[thick] (a7.center) -- (d2.center); 
\draw[thick] (a8.center) -- (d3.center); 
\draw[thick] (b7.center) -- (f1.center); 
\draw[thick] (c5.center) -- (f2.center); 
\draw[thick] (d6.center) -- (f3.center); 
\draw[thick] (d7.center) -- (f4.center); 
\draw[thick] (d8.center) -- (f5.center);

\draw[->,red]   (0*\xc, 9*\yc) --  (0*\xc,8*\yc) node[midway,right] {$1$};
\draw[->,red]   (1*\xc, 10*\yc) --  (1*\xc,9*\yc) node[midway,right] {$1$};
\draw[->,red]   (3*\xc, 3*\yc) --  (3*\xc,4*\yc) node[midway,right] {$0$};
\draw[->,red]   (4*\xc, 5*\yc) --  (4*\xc,7*\yc) node[midway,right] {$3$};
\draw[->,red]   (5*\xc, 6*\yc) --  (5*\xc,7*\yc) node[midway,right] {$0$};

  \node[blue] (c1) at (2*\xc,5*\yc){$\bullet$}; 
  \node[blue] (c2) at (2*\xc,6*\yc){$\bullet$}; 
  \node[blue] (c3) at (2*\xc,7*\yc){$\bullet$}; 
  \node[blue] (c4) at (2*\xc,8*\yc){$\bullet$}; 
  \node[blue] (c5) at (2*\xc,9*\yc){$\bullet$}; 
  
\end{tikzpicture}
\]
\caption{The lace diagrams associated to the dimension vector $\ud=(8,7,5,9,5,8)$, and raising vectors $(0,1,\star,0,4,0)$ (top left), $(0,2,\star,0,3,0)$ (top right), $(0,1,\star,0,3,1)$ (bottom left), $(1,1,\star,0,3,0)$ (bottom right). In fact, these raising vectors are the solutions of (QIP'), hence these lace diagrams describe the maximal-dimensional components of $\Sigma_{\ud}$.}   
\label{fig:875958 raising vectors}
\end{figure}
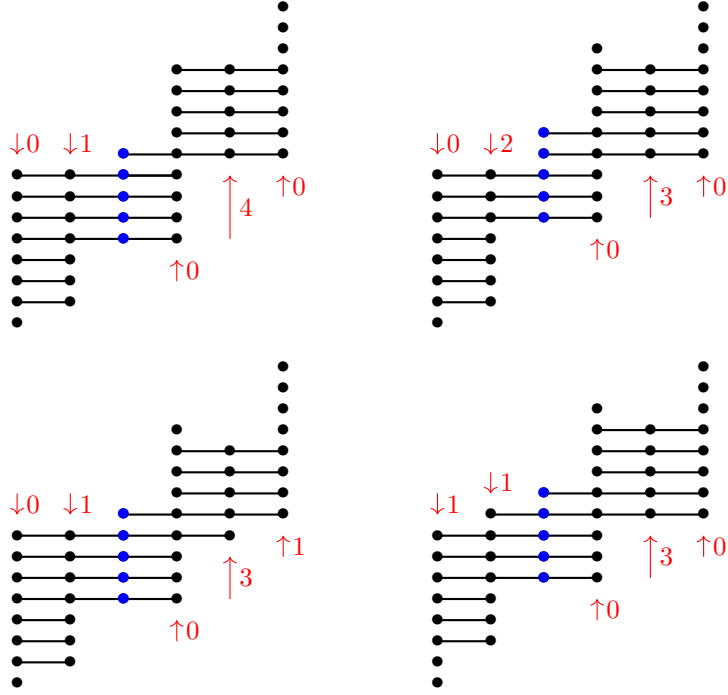

\begin{example} \rm
    Let $\ud=(8,7,5,9,5,8)$. The lace diagrams for $k=2$ and the raising vectors $(0,1,\star,0,4,0)$, $(0,2,\star,0,3,0)$, $(0,1,\star,0,3,1)$, $(1,1,\star,0,3,0)$ are shown in Figure~\ref{fig:875958 raising vectors}.
\end{example}

\begin{proposition}
	Let $\rv$ be a \rvector. The orbit $\Orep{\rv}$ is contained in $\Sigma_{\dv}$.
\end{proposition}
\begin{proof}
	The top dot in the first column of $\Dset{\rv}$  is $(0,d_k-1-e_0-\ldots -e_{k-1})$, while the bottom dot in the last column is $(n,e_{k+1}+\ldots + e_{n})$. We have $\sum e_i=d_k$, therefore there is no $y\in \Z$ such that $\Dset{\rv}$ has dots in both the first and the last columns at height $y$. Therefore, the multiplicity of the indecomposable module $\Irred{0n}$ in $\Vrep{\rv}$ is zero and $\Orep{\rv}\subset\Sigma_{\dv}$.
\end{proof}
We consider the following quadratic integer program on the set of raising vectors
$\rv=(e_0,\ldots, e_{k-1}, \star, e_{k+1} \ldots, e_n).$
\begin{align}
\tag{QIP'} \label{eq:1}
& \begin{array}{rl} \min & \displaystyle \QIPF(\rv)=\sum_{i\neq k}e_i\cdot(d_i-d_k) +\sum_{i\le j} e_ie_j\,, \\[.1in] \text{s.t.} & \displaystyle \ e_i\in \N, \quad \sum_{i\neq k} e_i=d_k.
\end{array}
\end{align}
\begin{remark} \label{rem:QIP} \rm
	The quadratic integer programs \eqref{eq:QIP} and \eqref{eq:1} are equivalent. They have the same minimum. Moreover, let $\sigma$ be a permutation of the set $\{0,\ldots, n\}$ such that $\sigma(k)=0$ and $d_{i}=d'_{\sigma(i)}$ for every $i$. Then $\sigma$ induces a bijection between the sets of solutions, i.e. $(e_1,\dots,e_n)$ is a solution of \eqref{eq:QIP} if and only if $(e_{\sigma(1)},\dots,e_{\sigma(k-1)},\star,e_{\sigma(k+1)},\dots,e_{\sigma(n)})$ is a solution of \eqref{eq:1}.
\end{remark}
We can now state our main result. Recall that the closures of the minimal codimensional orbits in $\Sigma_{\ud}$ are exactly the minimal codimensional components of $\Sigma_{\ud}$. 
\begin{theorem} \label{tw:main}
	The assignment $\rv \mapsto \Orep{\rv}$ induces a bijection
	$$\Orep{-}\colon\{\text{solutions of  \eqref{eq:1}}\} \to \{\text{minimal codimension orbits in $\Sigma_{\dv}$}\}. $$
\end{theorem}
\begin{corollary}
	Due to Remark \ref{rem:QIP}, the above theorem gives a bijection between the set of solutions of \eqref{eq:QIP} and the set of minimal codimension orbits in $\Sigma_{\dv}$.
\end{corollary}
 To prove the above theorem, we need two propositions whose proofs we relegate to the next section.
\begin{proposition} \label{pro:res1}
	Let $\rv$ be a \rvector. We have 
	$$\dim\Ext(\Vrep{\rv},\Vrep{\rv})=\QIPF(\rv).$$ 
\end{proposition}
\begin{proposition}\label{pro:res2}
	Let $\underline{a}$ and $\underline{b}$ be \rvector s. Suppose that $\Vrep{\underline{a}}\simeq \Vrep{\underline{b}}$. Then $\underline{a}=\underline{b}$.
\end{proposition}
\begin{proof}[Proof of Theorem \ref{tw:main}]
	Let $\rv$ be a solution of \eqref{eq:1}. By Proposition \ref{pro:res1} and Lemma \ref{lem:voight}, we have 
	$$\codim_{\Rep_{\dv}} \left( \Orep{\rv} \right)= \dim\Ext(\Vrep{\rv},\Vrep{\rv})=\QIPF(\rv).$$
	Due to Remark \ref{rem:QIP}, this is also the minimum of \eqref{eq:QIP}. Theorem \ref{thm:LR2} implies that the orbit $\Orep{\rv}$ is of minimal codimension. It follows that the map $\Orep{-}$ is well defined. By Proposition \ref{pro:res2}, it is injective. Moreover, by Theorem \ref{thm:LR2} and Remark \ref{rem:QIP} its domain and codomain have the same number of elements. Thus, the map $\Orep{-}$ is a bijection.
\end{proof}
\section{Proof}
\subsection{The case $k=0$}
 Fix a dimension vector $\dv=(d_0,\ldots, d_n)$ such that $d_0=\min(d_0,\ldots, d_n)$. Set $k=0$. We treat this case separately because the proofs are less technical and their main idea is already apparent. Moreover, in the general case we will refer to the case $k=0$. Before starting the proofs, let us first make some general observations. \\
 \\
 Let $\rv$ be a \rvector. Split the lace diagram $\Dset{\rv}$ into two parts: one consisting of the dots with the coordinate $y$ smaller than $d_0$ and the other consisting of the remaining dots. We use the notation $A=\Vrep{\rv}_{[0,d_0)}$ and $B=\Vrep{\rv}_{[d_0,\infty)}$ and obtain the corresponding decomposition
 $$\Vrep{\rv}=A\oplus B.$$
 The summand $A$
 is easy to control, while
 $B$
 is more complicated. More precisely, we have
 \begin{align} \label{eq:2}
 A=\bigoplus_{i=1}^n e_i \Irred{0,i-1}\,, \qquad B=\bigoplus_{i,j>0} c_{i,j} \Irred{i,j}\,,
 \end{align} 
 for some nonnegative integers $c_{i,j}$.
 \begin{example} \rm
 	The schematic diagram below illustrates the lace diagram $\Dset{\rv}$. It is divided into two parts. The blue part is the lace diagram corresponding to the module $A$ and the red part corresponds to the module $B$. The vertical segment on the left indicates that $k=0$.
    $$   
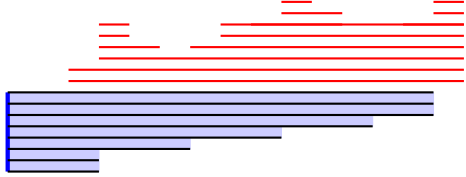
\begin{tikzpicture}
	
	\filldraw[color=blue!20, fill=blue!20] (0*\xx,7*\yy) -- (0*\xx,0*\yy) -- (3*\xx,0*\yy) -- (3*\xx,2*\yy)-- (6*\xx,2*\yy) -- (6*\xx,3*\yy) -- (9*\xx,3*\yy)
	-- (9*\xx,4*\yy)-- (12*\xx,4*\yy)-- (12*\xx,5*\yy) -- (14*\xx,5*\yy)-- (14*\xx,7*\yy);
	
	\draw[ultra thick,blue]  (0*\xx,0*\yy) --  (0*\xx,7*\yy);
	
	\draw[thick] (0*\xx,0*\yy) -- (3*\xx,0*\yy); 
	\draw[thick] (0*\xx,1*\yy) -- (3*\xx,1*\yy); 
	\draw[thick] (0*\xx,2*\yy) -- (6*\xx,2*\yy); 
	\draw[thick] (0*\xx,3*\yy) -- (9*\xx,3*\yy); 
	\draw[thick] (0*\xx,4*\yy) -- (12*\xx,4*\yy); 
	\draw[thick] (0*\xx,5*\yy) -- (14*\xx,5*\yy); 
	\draw[thick] (0*\xx,6*\yy) -- (14*\xx,6*\yy); 
	\draw[thick] (0*\xx,7*\yy) -- (14*\xx,7*\yy);
	
	\draw[thick,red] (2*\xx,8*\yy) -- (15*\xx,8*\yy);
	\draw[thick,red] (2*\xx,9*\yy) -- (15*\xx,9*\yy);
	\draw[thick,red] (3*\xx,10*\yy) -- (15*\xx,10*\yy);
	\draw[thick,red] (3*\xx,11*\yy) -- (5*\xx,11*\yy); \draw[thick,red] (6*\xx,11*\yy) -- (15*\xx,11*\yy);
	\draw[thick,red] (3*\xx,12*\yy) -- (4*\xx,12*\yy); \draw[thick,red] (7*\xx,12*\yy) -- (15*\xx,12*\yy);
	\draw[thick,red] (3*\xx,13*\yy) -- (4*\xx,13*\yy); \draw[thick,red] (7*\xx,13*\yy) -- (15*\xx,13*\yy);
	\draw[thick,red] (8*\xx,13*\yy) -- (11*\xx,13*\yy); \draw[thick,red] (13*\xx,13*\yy) -- (15*\xx,13*\yy);
	\draw[thick,red] (9*\xx,14*\yy) -- (11*\xx,14*\yy); \draw[thick,red] (14*\xx,14*\yy) -- (15*\xx,14*\yy);
	\draw[thick,red] (9*\xx,15*\yy) -- (10*\xx,15*\yy); \draw[thick,red] (14*\xx,15*\yy) -- (15*\xx,15*\yy);
	
\end{tikzpicture}
$$
 \end{example} 
 \begin{proof} [Proof of Proposition \ref{pro:res2} for $k=0$]
 	Let $\rv$ be any \rvector. The indecomposable quiver module $\Irred{0,i-1}$ occurs only in the summand
 	$A$.
 	Hence, the coordinate $e_i$ is equal to the multiplicity of the indecomposable module $\Irred{0,i-1}$ in $\Vrep{\rv}$.	
 \end{proof}
 \begin{example} \rm
 	Consider the case $n=3$ and $k=0$. The Kostant partition of $\Vrep{\rv}$ is of the form
 	$$
 	\begin{bmatrix}
 		e_1 & e_2 & e_3 & 0 \\
 		 & * & * & * \\
 		 &  & * & * \\
 		 & & & *
 	\end{bmatrix}.
 	$$
 \end{example}
 \begin{proof} [Proof of Proposition \ref{pro:res1} for $k=0$]
 	
 	The module
 	$B$
 	satisfies the assumption of Corollary~\ref{cor:OpenOrbit}, thus
 	$$\Ext(B,B)=0.$$
 	From Corollary~\ref{cor:Ext} and the description of indecomposable components \eqref{eq:2} of $A$ we obtain
 	$$\Ext(\Vrep{\rv},A)=0.$$
 	It follows that
 	$$\dim\Ext(\Vrep{\rv},\Vrep{\rv})=\dim\Ext(A,B)=
 	\sum_{i=1}^n e_i\cdot \dim\Ext(\Irred{0,i-1},B).$$
 	Fix an integer $i$. Let $\Irred{}\subset B$ be an indecomposable submodule such that the space $\Ext(\Irred{0,i-1},\Irred{})$ is nonzero. By \eqref{eqn:ExtIndecomposables} its lace diagram contains a dot in column $i$, i.e. $\Irred{}\simeq \Irred{ab}$ such that $a\le i\le b$. Moreover, every dot in column $i$ of the diagram of $B$
 	is a part of such submodule, because the diagram of $B$
 	has no dots in column $0$, cf. \eqref{eq:2}. Thus, the dimension of the space $\Ext(\Irred{0,i-1},B)$ is equal to the number of dots in column $i$ of the diagram of $B$
 	$$\dim\Ext(\Irred{0,i-1},B)=d_i-(d_0-e_1-e_2-\ldots-e_i).$$
 	Therefore we have
 	$$
 	\dim\Ext(\Vrep{\rv},\Vrep{\rv})= \sum_{i=1}^n e_i\cdot \left(d_i-d_0+\sum^{i}_{j=1}e_j\right)=\QIPF(\rv).
 	$$
 \end{proof}
 
\subsection{The general case} Fix a dimension vector $\dv=(d_0,\ldots, d_n)$ and an integer $k\in[0,n]$ such that $d_k=\min(d_0,\ldots, d_n).$ As before, let us first make some general observations. \\
\\
Let $\rv$ be a \rvector. Split the lace diagram $\Dset{\rv}$ into four parts depending on their coordinate $y$ and consider the corresponding decomposition of $\Vrep{\rv}$
\begin{align*}
	B_1&=\Vrep{\rv}_{[d_k,\infty)}\,,& A_1&=\Vrep{\rv}_{[\sum^{n}_{i=k+1}e_i,d_k)}   \,,&  A_2&=\Vrep{\rv}_{[0,\sum^{n}_{i=k+1}e_i)}\,,&  B_2&=\Vrep{\rv}_{(-\infty,0)}\,,
\end{align*}
$$\Vrep{\rv}=B_1 \oplus A_1\oplus A_2 \oplus B_2.$$
The summands $A_1$ and $A_2$ are easy to control, while
$B_1$ and $B_2$
are complicated, i.e.
\begin{align} \label{eq:3}
	A_1&=\bigoplus_{i=0}^{k-1} e_i \Irred{i+1,n} \,,& A_2&=\bigoplus_{i=k+1}^n e_i \Irred{0,i-1} \,,\\ 
	\nonumber B_1&=\bigoplus_{i,j>k} c_{i,j} \Irred{i,j}\,,&
	B_2&=\bigoplus_{i,j<k} c_{i,j} \Irred{i,j}\,,
\end{align} 
for some nonnegative integers $c_{i,j}$.
 \begin{example} \rm
	The schematic diagram below illustrates the lace diagram $\Dset{\rv}$. It is divided into four parts. The red parts correspond to the modules $B_1$ and $B_2$, while the blue ones correspond to $A_1$ and $A_2$. Explicitly: $B_1$ - the top red part, $B_2$ - the bottom red part, $A_1$ - the upper blue part, $A_2$ - the lower blue part. The vertical segment marks the position $k$, at the shortest vertical line.

    
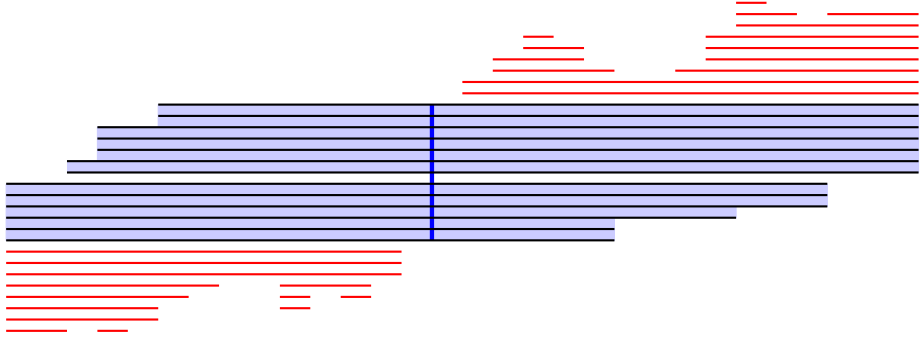
\[
\begin{tikzpicture}

\filldraw[color=blue!20, fill=blue!20] (0*\xx,0*\yy) -- (0*\xx,-5*\yy) -- (20*\xx,-5*\yy) -- (20*\xx,-3*\yy)
-- (24*\xx,-3*\yy) -- (24*\xx,-2*\yy) -- (27*\xx,-2*\yy)-- (27*\xx,0*\yy);
\filldraw[color=blue!20, fill=blue!20] (2*\xx,1*\yy)-- (30*\xx,1*\yy) -- (30*\xx,7*\yy) -- (5*\xx,7*\yy) -- (5*\xx,5*\yy) -- (3*\xx,5*\yy) -- (3*\xx,2*\yy)-- (2*\xx,2*\yy);

\draw[ultra thick,blue]  (14*\xx,7*\yy) --  (14*\xx,-5*\yy);

  \draw[thick] (2*\xx,1*\yy) -- (30*\xx,1*\yy); 
  \draw[thick] (2*\xx,2*\yy) -- (30*\xx,2*\yy); 
  \draw[thick] (3*\xx,3*\yy) -- (30*\xx,3*\yy); 
  \draw[thick] (3*\xx,4*\yy) -- (30*\xx,4*\yy); 
  \draw[thick] (3*\xx,5*\yy) -- (30*\xx,5*\yy); 
  \draw[thick] (5*\xx,6*\yy) -- (30*\xx,6*\yy); 
  \draw[thick] (5*\xx,7*\yy) -- (30*\xx,7*\yy); 

  \draw[thick,red] (15*\xx,8*\yy) -- (30*\xx,8*\yy);  
  \draw[thick,red] (15*\xx,9*\yy) -- (30*\xx,9*\yy); 
  \draw[thick,red] (16*\xx,10*\yy) -- (20*\xx,10*\yy); \draw[thick,red] (22*\xx,10*\yy) -- (30*\xx,10*\yy); 
  \draw[thick,red] (16*\xx,11*\yy) -- (19*\xx,11*\yy); \draw[thick,red] (23*\xx,11*\yy) -- (30*\xx,11*\yy); 
  \draw[thick,red] (17*\xx,12*\yy) -- (19*\xx,12*\yy); \draw[thick,red] (23*\xx,12*\yy) -- (30*\xx,12*\yy);  
  \draw[thick,red] (17*\xx,13*\yy) -- (18*\xx,13*\yy); \draw[thick,red] (23*\xx,13*\yy) -- (30*\xx,13*\yy);    
   \draw[thick,red] (24*\xx,14*\yy) -- (30*\xx,14*\yy);  
  \draw[thick,red] (24*\xx,15*\yy) -- (26*\xx,15*\yy); \draw[thick,red] (27*\xx,15*\yy) -- (30*\xx,15*\yy);  
  \draw[thick,red] (24*\xx,16*\yy) -- (25*\xx,16*\yy);

  \draw[thick] (0*\xx, 0*\yy) -- (27*\xx, 0*\yy); 
  \draw[thick] (0*\xx,-1*\yy) -- (27*\xx,-1*\yy); 
  \draw[thick] (0*\xx,-2*\yy) -- (27*\xx,-2*\yy); 
  \draw[thick] (0*\xx,-3*\yy) -- (24*\xx,-3*\yy); 
  \draw[thick] (0*\xx,-4*\yy) -- (20*\xx,-4*\yy); 
  \draw[thick] (0*\xx,-5*\yy) -- (20*\xx,-5*\yy); 

  \draw[thick,red] (0*\xx,-6*\yy) -- (13*\xx,-6*\yy);
  \draw[thick,red] (0*\xx,-7*\yy) -- (13*\xx,-7*\yy);
  \draw[thick,red] (0*\xx,-8*\yy) -- (13*\xx,-8*\yy);
  \draw[thick,red] (0*\xx,-9*\yy) -- (7*\xx,-9*\yy); \draw[thick,red] (9*\xx,-9*\yy) -- (12*\xx,-9*\yy); 
  \draw[thick,red] (0*\xx,-10*\yy) -- (6*\xx,-10*\yy); \draw[thick,red] (9*\xx,-10*\yy) -- (10*\xx,-10*\yy); \draw[thick,red] (11*\xx,-10*\yy) -- (12*\xx,-10*\yy);
  \draw[thick,red] (0*\xx,-11*\yy) -- (5*\xx,-11*\yy); \draw[thick,red] (9*\xx,-11*\yy) -- (10*\xx,-11*\yy);   
  \draw[thick,red] (0*\xx,-12*\yy) -- (5*\xx,-12*\yy);   
  \draw[thick,red] (0*\xx,-13*\yy) -- (2*\xx,-13*\yy);  \draw[thick,red] (3*\xx,-13*\yy) -- (4*\xx,-13*\yy);
\end{tikzpicture}
\]
\end{example}

\begin{proof} [Proof of Proposition \ref{pro:res2}]
	Let $\rv=(e_0,\ldots, e_{k-1}, \star, e_{k+1} \ldots, e_n)$ be a \rvector \ and $i$ an integer. 
    
	Suppose that $i>k$. Then, the indecomposable module $\Irred{0,i-1}$ occurs only in the summand
	$A_2$.
	The coordinate $e_i$ is the multiplicity of the indecomposable module $\Irred{0,i-1}$ in $\Vrep{\rv}$. 
    
	Consider now the other case, $i<k$. Then, the indecomposable quiver module $\Irred{i+1,n}$ occurs only in the summand	$A_1$. The coordinate $e_i$ is the multiplicity of the indecomposable module $\Irred{i+1,n}$ in $\Vrep{\rv}$.
\end{proof}
\begin{example} \rm
Consider the cases $n=5$ and $k=1$ and $k=2$. The Kostant partitions of $\Vrep{\rv}$ are of the form
$$
\begin{bmatrix}
	* &e_2&e_3&e_4&e_5& 0 \\
	  & 0 & 0 & 0 & 0 & e_0 \\
	  &   & * & * & * & * \\
  	  &   &   & * & * & * \\
	  &   &   &   & * & * \\
	  &   &   &   &   & *
\end{bmatrix}\,,
\qquad
\begin{bmatrix}
	* & * & e_3 & e_4 & e_5 & 0 \\
	& * & 0 & 0 & 0 & e_0 \\
	&   & 0 & 0 & 0 & e_1 \\
	&   &   & * & * & * \\
	&   &   &   & * & * \\
	&   &   &   &   & *
\end{bmatrix}.
$$
\end{example}
\begin{proof} [Proof of Proposition \ref{pro:res1}]
	We proceed analogously as in the previous section.
	The quiver modules
	$B_1\oplus A_1$ and $B_2 \oplus A_2$
	satisfy the assumptions of Corollary \ref{cor:OpenOrbit}, thus
	$$\Ext(B_1\oplus A_1,B_1\oplus A_1)=0\,, \qquad \Ext(B_2\oplus A_2,B_2\oplus A_2)=0.$$
	Directly from \eqref{eqn:ExtIndecomposables} and the description of indecomposable components \eqref{eq:3} we obtain
	$$\Ext(B_2,B_1)=0, \qquad\Ext(B_1,B_2)=0.$$
	Corollary \ref{cor:Ext} implies
	$$\Ext(A_1,\Vrep{\rv})=0, \qquad\Ext(\Vrep{\rv},A_2)=0.$$
	It follows that
	$$\Ext(\Vrep{\rv},\Vrep{\rv})=\Ext(B_2,A_1)\oplus\Ext(A_2,A_1)\oplus\Ext(A_2,B_1).$$
	The middle part can be computed directly
	\begin{align} \label{eq:proof1}
		\dim\Ext(A_2,A_1)=\sum_{i=k+1}^{n}\sum_{j=0}^{k-1} e_ie_j\cdot \dim\Ext(\Irred{0,i-1},\Irred{j+1,n})=\sum_{i>k>j} e_ie_j.
	\end{align}
	Let us focus on the part
	$$
	\dim\Ext(A_2,B_1)= \sum_{i=k+1}^n e_i\cdot \dim\Ext(\Irred{0,i-1},B_1).
	$$
	Reasoning from the previous section shows that the dimension of $\Ext(\Irred{0,i-1},B_1)$ is equal to the number of dots in column $i$ of $B_1$.
	We obtain
	\begin{align} \label{eq:proof2}
	\dim\Ext(A_2,B_1)= \sum_{i=k+1}^n e_i\cdot (d_i-d_k+\sum^{i}_{j=k+1}e_j).
	\end{align}
	Analogous reasoning shows that the dimension of $\Ext(B_2,\Irred{i+1,n})$ is equal to the number of dots in column $i$ of $B_2$, thus
	\begin{align} \label{eq:proof3}
		\dim\Ext(B_2,A_1)= \sum_{i=0}^{k-1} e_i\cdot (d_i-d_k+\sum^{k-1}_{j=i}e_j).
	\end{align}
	Summing formulas \eqref{eq:proof1},\eqref{eq:proof2} and \eqref{eq:proof3} we obtain the desired result.
\end{proof}

\bibliographystyle{alpha}
\bibliography{sample}

\end{document}